\documentclass[12pts]{article}

\usepackage[utf8]{inputenc}
\usepackage{graphics}
\usepackage{graphicx}
\usepackage{amssymb}
\usepackage{amsmath}
\usepackage{subfigure}
\usepackage{amsmath,amsbsy,amsfonts,amssymb}
\usepackage{latexsym,amssymb,amsmath,amsthm,amsfonts,enumerate,verbatim,xspace,
exscale}
\usepackage[T1]{fontenc}
\usepackage{amsfonts}
\usepackage[toc,page]{appendix}
\usepackage{cite}
\usepackage{apacite}
\usepackage{color}
\newtheorem{theorem}{Theorem}
\newtheorem{lemma}[theorem]{Lemma}

\newtheorem{proposition}[theorem]{Proposition}
\newtheorem{remark}[theorem]{Remark}

\newtheorem{example}[theorem]{Example}

\begin{document}

\title{Dimension reduction through Gamma convergence for general prestrained thin elastic sheets}
\author{David Padilla-Garza}
\maketitle
%\bibliographystyle{apacite}
%\bibliography{References}

%\begin{abstract}
%lorem ipsum
%\end{abstract}

\begin{abstract}
We study thin films with residual strain by analyzing the $\Gamma-$limit of non-Euclidean elastic energy functionals as the material's thickness tends to $0.$ We begin by extending prior results \cite{bhattacharya2016plates} \cite{agostiniani2018heterogeneous} \cite{lewicka2018dimension} \cite{schmidt2007plate},  to a wider class of films, whose prestrain depends on both the midplate and the transversal variables. The ansatz for our $\Gamma-$convergence result uses a specific type of wrinkling, which is built on exotic solutions to the Monge-Ampere equation, constructed via convex integration \cite{lewicka2017convex}. We show that the expression for our $\Gamma-$limit has a natural interpretation in terms of the orthogonal projection of the residual strain onto a suitable subspace. We also show that some type of wrinkling phenomenon is necessary to match the lower bound of the $\Gamma-$limit in certain circumstances. These results all assume a prestrain of the same order as the thickness; we also discuss why it is natural to focus on that regime by considering what can happen when the prestrain is larger. 
\end{abstract}

\section{Introduction}

We all know that a material tends to expand when heated. Expansion, or more generally a change in the stress-free metric of the material can also be caused by other factors. The object of study of this paper is a composite made of thin sheets of material with different stress free metrics. 

Thin elastic sheets that deform because of residual strain have recently been the focus of numerous mathematical \cite{agostiniani2018heterogeneous} \cite{bhattacharya2016plates} \cite{lewicka2011scaling} \cite{lewicka2018dimension} \cite{schmidt2007plate} \cite{lewicka2010foppl} \cite{lewicka2010shell} \cite{lewicka2011matching} and engineering/physics \cite{gladman2016biomimetic} \cite{miskin2018graphene} \cite{kim2012designing} \cite{klein2007shaping} \cite{pezzulla2015morphing} \cite{aharoni2018universal} \cite{sharon2002mechanics} studies. Mathematically, this represents the challenge of further generalizing the seminal work of \cite{friesecke2002theorem} to noneuclidean geometries, while from the point of view of applications, careful design of prestrain in thin sheets can be applied to 3D printing \cite{gladman2016biomimetic}. 

Other works which are not directly related, but still deal with prestrained thin sheets are \cite{efrati2009elastic}, which offers a physical treatment, and \cite{kupferman2014riemannian} and \cite{maor2019role}, which deal with thin elastic sheets immersed in Riemannian manifolds. 

Wrinkling in non-euclidean thin sheets is the focus of \cite{tobasco2021curvature}. In that paper, the author analyzes the wrinkling patterns in a thin spherical sheet confined to the surface of a liquid. Although in both this paper and \cite{tobasco2021curvature} the technique used is $\Gamma-$convergence, there are notable differences in both the approach and the result. In \cite{tobasco2021curvature}, the author assumes a geometrically linear von Kármán model in which the bending and membrane energy interact. We consider a general nonlinear elastic functional in which the membrane term dominates. The topology considered is also significantly different. 

From a mathematical perspective, the essential contribution of this paper is to generalize the work of \cite{agostiniani2018heterogeneous} and \cite{schmidt2007plate} to an arbitrary (non oscillatory) elastic energy and prestrain (as long as the metric is euclidean to leading order). We also show that this case can be reduced to one with a thickness independent elastic law and linear-in-thickness prestrain. We also analyze the optimality of the hypotheses. Unlike previous work, our lower bound needs to be complemented with an upper bound construction with fractional powers of $h$ to meet the energy. This ansatz involves the application of results from the literature that were proved using convex integration. We note that convex integration has also been used in the study of isometric  immersions, the Monge Ampere equation \cite{lewicka2017convex}, and fluid dynamics (see for example \cite{de2017high}).   

% An unexpected and potentially important feature of our analysis is the following: we identify situations in which the sheet will be macroscopically flat but microscopically wrinkled, and show that the wrinkles are necessary to relieve the prestrain, we also show that outside of this regime, wrinkling is not needed in order to meet the lower bound. Our analysis is mathematically rigorous and model independent (in the sense that it holds for any elastic energy). This section is ongoing work and will be published in a separate paper. 
 
%In related work that is still in progress, we identify a region in parameter space in which the sheet is macroscopically flat but we would expect to see the type of wrinkling that we use for our ansatz. Another question to be discussed elsewhere is this: how many sheets are needed to reconstruct an arbitrary surface in space? This question was approached in \cite{van2017growth}, our goal will be to provide a more general and rigorous treatment.   
% 
%  Another specific feature of our analysis is this: we show that two layers are enough to reconstruct any surface in 3D space up to isometry, assuming complete control over the stress-free metric. This principle was proved in a non mathematically rigorous way assuming a homogeneous isotropic linearly elastic energy in \cite{van2017growth}. This section is also ongoing work and will be published in a separate paper.
 
 The structure of the paper is as follows: we begin with the statement of our results and some introductory remarks. In section 3, we prove preliminary results which are very close to ones present in the literature: compactness and a lower bound for a sheet with arbitrary prestrain and elastic law, and an upper bound for a sheet with thickness-independent elastic law, and prestrain satisfying a centering hypothesis. In section 4, we prove an upper bound for thin sheets in which the elastic law is arbitrary, and the prestrain equals the identity at leading order, but is otherwise arbitrary. It is in this section where we apply a result proved using convex integration, as mentioned earlier. In order to prove that our ansatz achieves the lower bound, it is also necessary to show that it is possible to glue such constructions with classical ones. We note that this construction only works in the case when the preferred metric is the identity to leading order. The $\Gamma$ limit in the case of an arbitrary preferred metric is an open question.  Next, in section 5 we prove that the resulting quadratic function is, up to an inevitable left over residual strain, equivalent to the quadratic function arising as the $\Gamma$ limit of a sheet with thickness-independent elastic law and linear-in-thickness prestrain. In section 6 we prove that an ansatz that blows up at the $h$ scale is necessary to relieve a wide class of strains, and we also identify a regime in which such oscillations do not take place. Lastly in section 7, we analyze whether the hypotheses of the theorem are optimal. We conclude that several pathologies may occur if any of the hypotheses are omitted, even though a great part of the conclusion may still hold. 
 
% In the last part of the paper, we apply the mathematics developed to deduce physical properties of thin sheets: we reproduce the two-layer principle in greater generality and rigor than the existing literature, argue that our analysis predicts high amplitude wrinkling (high amplitude in the sense that the amplitude-thickness ratio blows up), and also analyze a scaling symmetry of the steepest-descent dynamics of our system.  
 
 We wish to thank Marta Lewicka for suggesting an ansatz based on the upper bound for a von Kármán energy scaling.

\section{Setting and overview of results}

Let $\Omega \subset \mathbf{R}^{2}$ be an open, bounded, connected set with piece-wise $C^{1}$ boundary. Let
\begin{equation}
\Omega^{h}=\Omega \times \left( -\frac{h}{2},\frac{h}{2}  \right).
\end{equation} 
We will denote a point $x \in  \Omega^{h}$ by $x=(x_{1}, x_{2}, x_{3})=(x', x_{3}).$ We will first study functionals of the form 
\begin{equation}
\mathcal{E}^{h}(u^{h})= \frac{1}{h}\int_{\Omega^{h}} W(x',\frac{x_{3}}{h},\nabla u^{h}(x) (A^{h}(x))^{-1}) dx.
\end{equation} 
For $u^{h} \in W^{1,2}(\Omega^{h}, \mathbf{R}^{3}).$ This corresponds to a thin film in which the preferred metric is non-Euclidean, and varies with thickness. We assume that $A^{h}(x)$ is of the form
\begin{equation}
A^{h}(x)=\overline{A}(x')+hB(x',\frac{x_{3}}{h}),
\end{equation}
with $\overline{A} \in C^{\infty}(\overline{\Omega}, \mathbf{R}^{3\times 3}_{sym,pos})$ and $B \in L^{\infty}(\Omega^{1}, \mathbf{R}^{3\times 3}_{sym}).$ We will later specialize to the case $\overline{A}=Id$ (strictly speaking, we only have a $\Gamma$ convergence result in the case $\overline{A}=Id$). 

%Let
%\begin{equation}
%(\overline{A})_{2 \times 2}(x')=F(x') 
%\end{equation}
%%with with $F \in C^{\infty}(\overline{\Omega}, \mathbf{R}^{2\times 2}_{sym, pos}),$ 
%and 
%\begin{equation}
%(B(x))_{2 \times 2}=G(x)
%\end{equation}
%with with $G \in L^{\infty}(\Omega^{1}, \mathbf{R}^{2\times 2} 2}_{sym}).$ 
Let 
\begin{equation}
\begin{split}
\mathcal{G}^{h}(x', \frac{x_{3}}{h})&=(A^{h}(x', \frac{x_{3}}{h}))^{2},\\
\mathcal{G}(x')&=(\overline{A}(x'))^{2}.
\end{split}
\end{equation}
Let $y \in W^{2,2}(\Omega,\mathbf{R}^{3})$ be such that 
\begin{equation}
\nabla y^{T} \nabla y =\left( \overline{A}^{2}(x') \right)_{2\times 2},
\end{equation}
where we denote by $\nabla y$ the $3\times 2$ matrix of partial derivatives. Define the Cosserat vector $b(x')$ as 
\begin{equation}
b(x')=(\nabla y)(\mathcal{G}_{2\times 2})^{-1}[\mathcal{G}_{1,3},\mathcal{G}_{2,3}]^{T}+\frac{\sqrt{\det \mathcal{G}}}{\sqrt{\det \mathcal{G}_{2\times 2}}} \nu(x'),
\end{equation}
where $\nu(x')$ is the unit normal to the surface, $\nu(x')=\frac{\partial_{1}y \times \partial_{2}y}{\parallel \partial_{1}y \times \partial_{2}y \parallel}.$ The Cosserat vector satisfies
\begin{equation}
\left( \nabla y| b \right)^{T}\left( \nabla y| b \right)=\left( \overline{A}(x') \right)^{2}.
\end{equation}
We assume that the elastic law $W: \Omega^{1}\times \mathbf{R}^{3\times 3} \to \mathbf{R}$ satisfies 
\begin{itemize}
\item[i)] For a.e. $x \in \Omega^{1}$ the function $W(x,\cdot)$ is frame indifferent, that is
\begin{equation}
W(x,F)=W(x,RF)
\end{equation}
for every $F \in \mathbf{R}^{3\times 3}$ and $R \in SO(3).$

\item[ii)] For a.e. $x \in \Omega^{1},$ the energy $W(x,\cdot)$ is minimized at $SO(3),$ and the minimum is $0.$

\item[iii)] There exists a constant $c$ (independent of $x$) such that $ W(x,F) \geq c \ \mbox{dist}^{2}(F,SO(3)).$

\item[iv)] There exists a neighborhood $\mathcal{U}$ of $SO(3)$ such that for all $x \in \Omega^{1}$ we have $W(x,F)$ is $C^{2}$ regular in $F$ for $F \in \mathcal{U}.$ We also have that $D^{2}W(x,F)$ is uniformly equicontinuous in $F$ for all $F \in \mathcal{U},$ i.e. for all $\epsilon>0$ and $x \in \Omega^{1},$ there exists $\delta>0$ such that if $\parallel F -G \parallel<\delta$ then
\begin{equation}\label{equicontinuity}
\left| D^{2}W(x,F)-D^{2}W(x,G) \right| <\epsilon.
\end{equation}

\item[v)] There exists a constant $C$ such that
\begin{equation}
Q_{3}(x, F) \leq C \left| \text{sym} F \right|^{2}
\end{equation}
for all $x \in \Omega^{1}.$

\end{itemize}

If $M \in \mathbf{R}^{m_{1}\times m_{2}}$ and $N \in \mathbf{R}^{n_{1}\times n_{2}}$ with $m_{1}>n_{1}$ and $m_{2}>n_{2},$ we define the operation
\begin{equation}
M+N=M+\iota(N),
\end{equation}
where $\iota$ is the inclusion function from $\mathbf{R}^{n_{1}\times n_{2}}$ to $\mathbf{R}^{m_{1}\times m_{2}}$ defined as 
\begin{equation}
\iota(N)=\sum_{i,j} N_{ij}e_{i}\otimes e_{j}.
\end{equation}
%\item[v)] There exists a constant $C$ such that
%\begin{equation}
%\parallel D^{2}W(x,Id) \parallel \leq C
%\end{equation}
%for all $x$ in some neighborhood of $Id.$

%Let $u^{h} \in W^{1,2}(\Omega^{h} \to \mathbf{R}^{3})$ be a sequence such that 
%\begin{equation}\label{hsquarebound}
%\mathcal{E}^{h}(u^{h}) \leq k h^{2}.
%\end{equation}
Let 
\begin{equation}
{Q}_{3}(x,F)=D^{2}W(x,\mbox{Id})(F,F),
\end{equation}
let ${L}(x)$ be the tensor such that be such that
\begin{equation}
{Q}_{3}(x,F)=\langle L(x)F,F \rangle.
\end{equation}
For $X \in \mathbf{R}^{2 \times 2}$ define the quadratic form 
\begin{equation}\label{definition of Q_2}
Q_{2}(x',t,X,\overline{A})=\min_{d \in \mathbf{R}^{3}} Q_{3}(x',t,\overline{A}^{-1}(x')[X+d \otimes e_{3}]\overline{A}^{-1}(x'))
\end{equation}
and for $X \in L^{2}((- \frac{1}{2},\frac{1}{2}), \mathbf{R}^{2 \times 2})$
\begin{equation}\label{definition of Q_2'}
Q_{2}'(x',X,\overline{A})=\min_{s \in \mathbf{R}^{ 2\times 2}} \int_{- \frac{1}{2}}^{\frac{1}{2}} Q_{2}(x',t,X + s,\overline{A}) \, dt.
\end{equation}

%For $H \in \mathcal{L}^{2}(\Omega, \mathbf{R}^{3\times 3}),$ define the form 
%\begin{equation}\label{limitquadratic}
%\mathcal{Q}(x',H)=\min_{s,d} \int_{-\frac{1}{2}}^{\frac{1}{2}} \mathcal{Q}_{3}\Big( x,\overline{A}^{-1}(x')[s+x_{3}H(x')-\overline{A}(x')B(x',x_{3}) + e_{3} \otimes d(x_{3})]\overline{A}^{-1}(x') \Big) dx_{3}
%\end{equation}
%%and the form
%%\begin{equation}
%%=\min_{M,c(t)} \int_{\Omega} \mathcal{Q}^{*}(x',H,M(x')+e_{3} \otimes c(x',t)) dx',
%%\end{equation}
%where $s$ runs over $s \in \mathbf{R}^{2 \times 2}$ and $d(t) \in L^{2}((0,1),\mathbf{R}^{3}).$ 
%Also define the form
%\begin{equation}
%\overline{\mathcal{Q}}(H)=\min_{d,g}\int_{\Omega^{1}} \mathcal{Q}_{3}\Big( x,\overline{A}^{-1}(x')[Q^{T}\nabla g+x_{3}H(x')-\\
%\overline{A}(x')B(x',x_{3}) + d(x',x_{3}) \otimes e_{3}]\overline{A}^{-1}(x') \Big)dx,
%\end{equation}
%where $c(x',t) \in L^{2}(\Omega^{1},\mathbf{R}^{3})$ and $g \in W^{1,2}(\Omega^{1},\mathbf{R}^{3})$ such that $\frac{\partial}{\partial x_{3}} g =0.$ We have used the notation $\nabla'y=(\partial_{x_{1}}y, \partial_{x_{2}}y).$

Let $y \in W^{2,2}(\Omega,\mathbf{R}^{3}).$ Define the functional ${\mathcal{I}}(y)$ as
\begin{equation}\label{Gammalimitfunctional}
{\mathcal{I}}(y)=
\begin{cases}
\frac{1}{2}\int_{\Omega^{1}} Q_{2}'(x', t \nabla {y}^{T} \nabla {b}-(\overline{A}B)_{2\times 2}, \overline{A}) dx', \mbox{ if } (\nabla y)^{T}(\nabla y)=\mathcal{G}_{2\times 2}\\
\infty \mbox{ if not.}
\end{cases}
\end{equation}

Before stating our results precisely we briefly review some of the existing results and describe how they relate to ours. The work of \cite{schmidt2007plate} considered plates whose elastic law and prestrain were independent of $x'.$ When specialized to that case, our treatment is equivalent to his. The work of \cite{agostiniani2018heterogeneous} considered prestrains that depend on $x_{3}$ as well as $x',$ but imposed the restriction
\begin{equation}\label{centeringcondition}
\text{curlcurl} \left( \int_{-\frac{1}{2}}^{\frac{1}{2}} B_{2\times 2}(x', x_{3}) \, dx_{3} \right) =0.
\end{equation}
The reference \cite{agostiniani2018heterogeneous} also took the elastic law to be independent of $x_{3}.$ The condition \eqref{centeringcondition} is not particularly natural, but it was needed in \cite{agostiniani2018heterogeneous} to give an ansatz that meets the lower bound (i.e. it was needed to prove that the $\Gamma-$liminf and the $\Gamma-$limsup agree). The most important development in this paper is that we do not assume a condition like \eqref{centeringcondition}, provided $\overline{A}(x')=Id.$ Also, unlike \cite{agostiniani2018heterogeneous} our elastic law can depend on $x_{3}$ as well as $x'.$

This development uses a new upper bound ansatz. 

We turn now to more precise statements of our results, specifically
\begin{itemize}
\item[$\bullet$] A lower bound (Theorem \ref{Gammaconvergencetheorem1}) which holds for any $\overline{A}(x').$

\item[$\bullet$] An upper bound (Proposition \ref{propositiontheorem}) that's directly analogous to that of  \cite{bhattacharya2016plates} (in particular, it does assume a condition like \eqref{centeringcondition}). 

\item[$\bullet$] A better upper bound (Theorem \ref{theorem2}), which matches the lower bound and therefore gives a $\Gamma-$convergence theorem when $\overline{A}=Id$ (but with no artificial condition like \eqref{centeringcondition}).
\end{itemize}

The proofs of Theorem \ref{Gammaconvergencetheorem1} and Proposition \ref{propositiontheorem} use tools similar to those of \cite{agostiniani2018heterogeneous} and \cite{schmidt2007plate}, but the proof of Theorem \ref{theorem2} is different: as mentioned in the introduction, it uses a wrinkling ansatz from \cite{lewicka2017convex}, which was obtained using convex integration. 

\begin{theorem} \label{Gammaconvergencetheorem1}
Let $u^{h} \in W^{1,2}(\Omega^{h} \to \mathbf{R}^{3})$ be a sequence such that 
\begin{equation}\label{hsquarebound}
\mathcal{E}^{h}(u^{h}) \leq k h^{2}.
\end{equation}
Then
%\begin{equation}
%\mathcal{E}^{h}(u^{h}) \leq k h^{2}. 
%\end{equation}
\begin{itemize}
\item[i] Compactness. There exist $c^{h} \in \mathbf{R}^{3}$ and $Q^{h} \in SO(3)$ such that for the renormalized deformations
\begin{equation}
y^{h}(x',x_{3})=Q^{h}u^{h}(x',\frac{x_{3}}{h})-c^{h}
\end{equation}
we have 
\begin{equation}
y^{h} \to y 
\end{equation}
strongly in $W^{1,2}(\Omega^{1}, \mathbf{R}^{3}),$ for some $y\in W^{2,2}(\Omega^{1}, \mathbf{R}^{3})$ independent of $x_{3}.$ With a slight abuse of notation we treat $y$ interchangeably as a function defined on $\Omega$ or $\Omega^{1}.$ The function $y$ satisfies that $(\nabla y)^{T}\nabla y=(\overline{A}^{2}(x'))_{2\times 2}.$ We also have
\begin{equation}
\frac{1}{h} \partial_{3}y^{h} \to b
\end{equation}
strongly in $L^{2},$ where $b$ is the Cosserat vector. 

\item[ii] Lower bound:
\begin{equation}
\liminf_{h \to 0} \frac{1}{h^{2}} \mathcal{E}^{h}(u^{h}) \geq {\mathcal{I}}(y).
\end{equation} 

\end{itemize}

\end{theorem}

A matching upper bound also holds, under additional hypotheses. We state it as a proposition since it does not require any ideas other than the ones already present in the literature. 
\begin{proposition}\label{propositiontheorem}
 Upper bound. Assume $W(x',x_{3})=W(x'),$ and
\begin{equation}
\int_{-\frac{1}{2}}^{\frac{1}{2}} B_{2\times 2}(x',t)dt=0
\end{equation}
and let $y \in W^{2,2}(\Omega, \mathbf{R}^{3})$ 
%and assume that 
%\begin{equation}
%{\textbf{curl}} {\textbf{curl}} M_{min}=0,
%\end{equation} 
then there exists a sequence $u^{h}(x)$ such that, for the renormalized  sequence $y^{h}(x',x_{3})=u^{h}(x',\frac{x_{3}}{h})$ and $\nabla^{h}y^{h}=\left(\nabla'y^{h}, \frac{1}{h}\partial_{3}y^{h} \right)$ we have 
\begin{equation}
\nabla^{h} y^{h} \to (\nabla' y, b)
\end{equation}
strongly in $W^{1,2}(\Omega^{1}, \mathbf{R}^{3})$ (identifying $y$ with its trivial extension in $\Omega^{1}$) and
\begin{equation}
\lim_{h \to 0} \frac{1}{h^{2}} \mathcal{E}^{h}(y^{h})={\mathcal{I}}(y),
\end{equation}
where
\begin{equation}
\mathcal{E}^{h}(y^{h})= \int_{\Omega^{h}} W(x',\frac{x_{3}}{h},\nabla^{h} y^{h}(x) (A^{h}(x))^{-1}) dx.
\end{equation} 
\end{proposition}

The following result does better than Proposition 2, in the sense that it requires no centering condition like \eqref{centeringcondition}. However, it is restricted to the case $\overline{A}(x')=Id.$

\begin{theorem}\label{theorem2}
Assume $\overline{A}=Id,$ and let $y \in W^{2,2}(\Omega, \mathbf{R}^{3})$ 
%and assume that 
%\begin{equation}
%{\textbf{curl}} {\textbf{curl}} M_{min}=0,
%\end{equation} 
then there exists a sequence $u^{h}(x)$ such that, for the renormalized  sequence $y^{h}(x',x_{3})=u^{h}(x',\frac{x_{3}}{h})$ and $\nabla^{h}y^{h}=\left(\nabla'y^{h}, \frac{1}{h}\partial_{3}y^{h} \right)$ we have 
\begin{equation}
\nabla^{h} y^{h} \to (\nabla' y, \nu)
\end{equation}
strongly in $W^{1,2}(\Omega^{1}, \mathbf{R}^{3})$ (identifying $y$ with its trivial extension in $\Omega^{1}$) and
\begin{equation}
\lim_{h \to 0} \frac{1}{h^{2}} \mathcal{E}^{h}(y^{h})={\mathcal{I}}(y).
\end{equation}
\end{theorem}

In order to have finite energy at order $1,$ the limiting deformation must achieve the metric $\overline{A}^{2},$ hence it is obligated that the ansatz starts with a term $y(x')+hx_{3}b(x').$ However, unlike previous works our ansatz includes terms of order $h^{\frac{1}{2}}.$ This raises the question of whether it is possible to achieve the lower bound with an ansatz $y^{h}$ such that 
\begin{equation}\label{hcompactness}
\parallel y^{h}-\left( y(x')+hx_{3}b(x') \right) \parallel_{W^{1,2}} \leq Ch
\end{equation}
In section 6 we prove that this is not possible (in fact, we prove a slightly stronger result). By doing so, we show that relieving an arbitrary strain implies a deformation that blows up at the $h$ scale. 

These results all deal with a prestrain whose variation in $x_{3}$ is of order $h.$ In physical terms, this is reasonable since it means that the prestrain is of the same order as the thickness. In section 7 we investigate whether a sheet being in the bending regime, i.e.
\begin{equation}\label{bendingregime}
\limsup_{h \to \infty} \frac{1}{h^{2}} \mathcal{E}^{h}(y^{h}) < \infty
\end{equation} 
implies that $A^{h}(x',x_{3})=\overline{A}(x')+hB^{h}(x',x_{3}),$ where $B^{h}(x',x_{3})$ is bounded (in some $L^{p}$ norm). We show that this is not true, not even if the hypotheses are significantly strengthened. However, we show that \eqref{bendingregime} implies that 
\begin{equation}
\left( \left( A^{h} (x',x_{3}) \right)^{2} \right)_{2\times 2} \to \overline{A}^{2}(x').
\end{equation}
In other words, finite bending energy does not imply finite prestrain, but it does imply that the metric is thickness-independent to leading order. 

In the cases treated in Theorems \ref{Gammaconvergencetheorem1} and \ref{theorem2} the limiting energy is a quadratic form of the generalized second fundamental form $\nabla {y}^{T} \nabla {b},$ which can be written as the integral 
\begin{equation}\label{limitingquadraticform}
\frac{1}{2}\int_{\Omega^{1}} Q_{2}'(x', t\nabla {y}^{T} \nabla {b}-(\overline{A}B )_{2\times 2}, \overline{A}) dx',
\end{equation}
where $Q_{2}'$ is given by \eqref{definition of Q_2'}. In section 5 we show that problem \eqref{limitingquadraticform} can be simplified to a thickness independent quadratic form, and linear-in-thickness prestrain, in other words
\begin{multline}
\int_{\Omega^{1}} Q_{2}'(x', t\nabla {y}^{T} \nabla {b}-(\overline{A}B)_{2\times 2}, \overline{A}) dx' =\\
\int_{\Omega^{1}} Q_{2}^{*}(x', \nabla {y}^{T} \nabla {b}-(\overline{A}B)_{2\times 2}^{*}(x')) dx'+E(Q_{2}, B),
\end{multline}
with explicit expressions for $E(Q_{2}, B, \overline{A}), Q_{2}^{*}, B^{*}.$

\section{Proofs of Theorem 1 (and Proposition 2)}

\subsection{Compactness}

This section follows the work of \cite{lewicka2018dimension}. We will only write an outline of the main ideas, and refer to \cite{lewicka2011scaling} and \cite{bhattacharya2016plates} for the full argument. Using the fact that $\left| \overline{A}(x')\right|$ is bounded above and away from $0,$ we get that $\parallel \nabla^{h}y^{h} \parallel_{L^{2}} \leq K_{1}.$ In order to see this, note that since we are in the bending regime,
\begin{equation}
 \int_{\Omega^{1}} \mbox{dist}^{2}(\nabla^{h} y^{h}({A}^{h}(x))^{-1},\mbox{SO}(3)) dx  \leq C h^{2}.
\end{equation}
Hence, there exists a measurable rotation field $R(x): \Omega^{1} \to SO(3)$ such that
\begin{equation}
\int_{\Omega^{1}} \left| \nabla^{h}y^{h} (A^{h}(x))^{-1} -R(x)\right|^{2} dx \leq C h^{2}.
\end{equation}
Therefore
\begin{equation}
 \parallel \nabla^{h}y^{h} (A^{h}(x))^{-1} -R(x) \parallel_{L^{2}} \leq C h.
\end{equation}
Then 
\begin{equation}
 \parallel \nabla^{h}y^{h} (A^{h}(x))^{-1} \parallel_{L^{2}} \leq  \parallel R(x) \parallel_{L^{2}} + C h.
\end{equation}
This implies 
\begin{equation}
 \parallel \nabla^{h}y^{h} (A^{h}(x))^{-1} \parallel_{L^{2}} \leq k.
\end{equation}
Using the hypothesis that $\overline{A}$ is bounded above and away from $0,$ and that $B^{h}$ is uniformly bounded, we get for $h$ small enough that
\begin{equation}
 \parallel \nabla^{h}y^{h} \parallel_{L^{2}} \leq k.
\end{equation}
 We also have that $\left| B(x',x_{3}) \right|$ is uniformly bounded. Using this, along with triangle inequality we get.
 
\begin{equation} \label{bound1}
\begin{split}
 \int_{\Omega^{1}} \mbox{dist}^{2}(\nabla^{h} y^{h}(\overline{A}(x'))^{-1},\mbox{SO}(3)) dx &\leq   2\int_{\Omega^{1}} \mbox{dist}^{2}(\nabla^{h}y^{h}(\overline{A}(x'))^{-1},\nabla^{h}y^{h}({A}^{h}(x))^{-1})) dx\\
&+2\int_{\Omega^{1}}\mbox{dist}^{2}(\nabla^{h}y^{h}({A}^{h}(x))^{-1},SO(3))) dx \\
&\leq C \int_{\Omega^{1}} \parallel \nabla^{h}y^{h} \parallel_{L^{2}}^{2} \left( \max_{x \in \Omega^{1}} \mbox{dist}^{2}((\overline{A}(x'))^{-1},{A}^{h}(x))^{-1})\right)+\\
& \	\	\	\	C\mathcal{E}^{h}(y^{h}) \\
&\leq Ch^{2}+C\mathcal{E}^{h}(y^{h}),
\end{split}
\end{equation} 
hence the results in \cite{lewicka2011scaling} and \cite{bhattacharya2016plates} yield compactess for the desired limit: defining 
\begin{equation}
\widetilde{\mathcal{E}}(u^{h})=\frac{1}{h}\int_{\Omega^{h}} W(x, \nabla u^{h} \overline{A}(x'))
\end{equation}
we get by \eqref{bound1} that 
\begin{equation}
\widetilde{\mathcal{E}}(u^{h}) \leq C h^{2}
\end{equation}
and therefore Lemma 2.3 and Theorem 2.1 $(i)$ of \cite{bhattacharya2016plates} imply the result.

\begin{remark}
The same proof would hold if instead of considering $A^{h}(x)=\overline{A}(x')+h(B(x',x_{3})),$ we consider 
\begin{equation}
A^{h}(x)=\overline{A}(x')+hB^{h}(x',x_{3}),
\end{equation}
where $B^{h} \to B$ strongly in $L^{\infty}(\Omega^{1}).$
\end{remark}

%By \cite{lewicka2011scaling} we get that there exist 

\subsection{Lower bound}

Before giving the proof of the lower bound, we need a technical lemma:
\begin{lemma} \label{lemmaweak}
Let $f^{h} \in L^{2}(\Omega^{1}, \mathbf{R}^{3 \times 3} )$ be such that 
\begin{equation*}
\frac{f^{h}}{h} \rightharpoonup f,
\end{equation*}
weakly in $L^{2}$ then
\begin{equation} \label{eqlemma}
\frac{1}{h^{2}}\liminf \int_{\Omega^{1}} W(x,\mbox{Id}+f^{h}(x))dx \geq \frac{1}{2} \int_{\Omega^{1}} Q_{3}(x,f) dx
\end{equation}
\end{lemma}

\begin{proof}
Let $F^{h}=\{x \in \Omega^{1}| f^{h}(x) \leq {h^{0.9}}\}.$ By hypothesis $ \limsup_{h} \parallel \frac{f^{h}}{h} \parallel_{L^{2}} < \infty,$ therefore $|\Omega^{1} \setminus F_{h}| \to 0,$ since 
\begin{equation}
\begin{split}
\parallel f^{h} \parallel_{L^{2}}^{2} &= \int_{\Omega^{1}} \left| f^{h} \right|^{2} dx \\
&\geq \int_{\Omega^{1} \setminus F^{h}} \left| f^{h} \right|^{2} dx\\
&\geq \int_{\Omega^{1} \setminus F^{h}} h^{1.8} dx\\
& = h^{1.8} |\Omega^{1}\setminus F^{h}|.
\end{split}
\end{equation}
Hence,
\begin{equation}
\frac{1}{h}\| {f^{h}} \|_{L^{2}} \geq h^{-0.1} \sqrt{|\Omega^{1}\setminus F^{h}|}.
\end{equation}
%Since 
%\begin{equation}
%\| \frac{f^{h}}{h} \|_{L^{2}}
%\end{equation}
%is bounded,
And so, we have that 
\begin{equation}
|\Omega^{1}\setminus F^{h}| \to 0.
\end{equation}
This implies 
\begin{equation}
\mathbf{1}_{F^{h}} \to \mathbf{1}_{\Omega^{1}}
\end{equation}
in $L_{2},$ since 
\begin{equation}
\begin{split}
\| \mathbf{1}_{F^{h}} - \mathbf{1}_{\Omega^{1}} \|_{L^{2}}^{2} &= |\Omega^{1}\setminus F^{h}| \\
 &\to 0.
\end{split} 
\end{equation}
By Taylor expanding $W(x, \cdot)$ at the identity, we have, for $x \in F^{h},$
\begin{equation}
\frac{1}{h^{2}}W(x,\mbox{Id}+f^{h}(x))= \frac{1}{2} Q_{3}\left(x, \frac{f^{h}}{h}(x) \right)+ \textit{o}(1) \left| \frac{f^{h}}{h}\right|^{2}.
\end{equation}

This is because,
\begin{equation}
\frac{1}{h^{2}}W(x,\mbox{Id}+f^{h}(x))= \frac{1}{2} Q_{3}\left(x, \frac{f^{h}}{h}(x) \right)+ \frac{1}{2} \left( D^{2}W(x, \xi)\left( \frac{f^{h}}{h} \right) - Q_{3}\left(x, \frac{f^{h}}{h}(x) \right) \right),
\end{equation}
where $\xi$ is in the line segment joining $Id$ and $Id + f^{h}.$ Since $W$ is $C^{2}$ in a neighborhood of the origin, and by hypothesis $|f^{h}(x)| \leq h^{0.9}$ in $F^{h}$ we have that
\begin{equation}
\left| D^{2}W(x, \xi)\left( \frac{f^{h}}{h} \right) - Q_{3}\left(x, \frac{f^{h}}{h}(x) \right) \right| = o(1) \left| \frac{f^{h}}{h} \right|^{2}
\end{equation} 
Using hypothesis iv), we can ensure that the error is $\textit{o}(1)$ uniformly in $x.$ Hence
\begin{equation}
\begin{split}
\frac{1}{h^{2}} \int_{\Omega^{1}} W(x,\mbox{Id}+f^{h}(x)) dx &\geq \frac{1}{h^{2}} \int_{F^{h}} W(x,\mbox{Id}+f^{h}(x)) dx\\
&= \frac{1}{2}\int_{F^{h}} Q_{3}\left(x, \frac{f^{h}(x)}{h} \right) dx+\parallel \frac{f^{h}}{h}\parallel_{L^{2}}^{2}o(1)\\
&= \frac{1}{2}\int_{\Omega^{1}} Q_{3}\left(x, \mathbf{1}_{F^{h}}\frac{f^{h}(x)}{h} \right) dx+\parallel \frac{f^{h}}{h}\parallel_{L^{2}}^{2}o(1).
\end{split}
\end{equation}
Now $\mathbf{1}_{F^{h}}\frac{f^{h}(x)}{h} \rightharpoonup f(x)$ by the weak-strong lemma, and $Q_{3}$ is lower semi continuous with respect to weak convergence, so we get \eqref{eqlemma}.

\end{proof}

We now prove the lower bound, for a general $\overline{A}.$

\begin{proof}

(Of Theorem \ref{Gammaconvergencetheorem1})

We begin with a compactness result for the re-normalized deformations. By \cite{bhattacharya2016plates}, \cite{lewicka2011scaling}, \cite{lewicka2018dimension} we have that there exists $SO(3)$ valued fields $R^{h}(x')$ such that
\begin{equation}\label{existRhfields}
 \int_{\Omega^{1}} \left| \nabla^{h}y^{h}\overline{A}^{-1}(x')-R^{h}(x')  \right|^{2} dx \leq Ch^{2}.
\end{equation}
We define the quantity $\overline{S}^{h}(x',x_{3})$ as
\begin{equation}\label{defoverlineSh}
\overline{S}^{h}(x',x_{3})=\frac{1}{h} \left( R^{h}(x')^{T}\nabla^{h}y^{h} \overline{A}^{-1}(x') -\mbox{Id} \right).
\end{equation}
Then, as $h \to 0$ we have that $\overline{S}^{h}(x',x_{3}) \rightharpoonup \overline{S}(x',x_{3})$ weakly in $L^{2},$ where $\overline{S}(x',x_{3})$ satisfies
\begin{equation}\label{defoverlineS}
\left( \overline{A}(x')\overline{S}(x',x_{3})\overline{A}(x') \right)_{2\times 2}=s(x')+x_{3}\nabla y(x')^{T} \nabla b(x'),
\end{equation}
where $b(x')$ is the Cosserat vector, for some $s \in L^{2}(\Omega, \mathbf{R}^{2 \times 2})$. We can define a similar quantity $S^{h}(x',x_{3})$ for the metric $A^{h}(x',x_{3})$ instead of $\overline{A}(x')$:
\begin{equation}\label{defSh}
{S}^{h}(x',x_{3})=\frac{1}{h} \left( R^{h}(x')^{T}\nabla^{h}y^{h}(A^{h})^{-1}(x',x_{3}) -\mbox{Id} \right).
\end{equation}
Then, as $h \to 0$ we have that ${S}^{h}(x',x_{3}) \rightharpoonup {S}(x',x_{3})$ weakly in $L^{2},$ where ${S}(x',x_{3})$ satisfies
\begin{equation}\label{defS}
\left( \overline{A}(x'){S}(x',x_{3})\overline{A}(x') \right)_{2\times 2}=s(x')+x_{3}\nabla y(x')^{T} \nabla b(x')-(\overline{A}B)_{2\times 2}.
\end{equation}
Let $d(x',t) \in L^{2}(\Omega^{1},\mathbf{R}^{3})$ be such that
\begin{multline}
\mbox{sym}\left( \overline{A}(x'){S}(x',x_{3})\overline{A}(x') \right)=\\
\mbox{sym} \left[ s(x')+x_{3}\nabla y(x')^{T} \nabla b(x')-(\overline{A}B)_{2\times 2}(x',x_{3}) +d(x',t)\otimes e_{3} \right] .
\end{multline}

%Note that $d \in L^{\infty}$ and we have the bound
%\begin{equation}\label{Linftyboundond}
%\parallel d(x',x_{3}) \parallel \leq \frac{C}{c} \left( \parallel s-B_{2\times 2} \parallel_{L^{\infty}} \right)+\parallel B_{2\times 2} \parallel_{L^{\infty}},
%\end{equation}
%where $c,C$ are such that
%\begin{equation}
%c \parallel A \parallel \leq Q_{3}(x,A) \leq C \parallel A \parallel. 
%\end{equation}
%To deduce equation \eqref{Linftyboundond}, note that we can write
%\begin{equation}
%\begin{split}
%Q_{3}\left( x, s-B_{2\times 2}+d\otimes e_{3} \right) &\geq c \parallel s(x')-B_{2\times 2}(x',x_{3})+d(x',x_{3})\otimes e_{3} \parallel\\
%&\geq c \left( |d(x',x_{3})| - \parallel B_{2\times 2} \parallel_{L^{\infty}} \right).
%\end{split}
%\end{equation}
%We also have
%\begin{equation}
%\begin{split}
%Q_{2}\left( x, s-B_{2\times 2}_{2\times 2} \right) &\leq Q_{3}\left( x, s-B_{2\times 2} \right)\\
%& \leq C \parallel s-B_{2\times 2} \parallel_{L^{\infty}}.
%\end{split}
%\end{equation}
%Combining these two estimates yields \eqref{Linftyboundond}.
Using frame indifference, we can write
\begin{equation}
\begin{split}
 \liminf \mathcal{E}^{h}(u^{h}) &=\liminf\frac{1}{h^{2}} \int_{\Omega^{1}} W(x,\nabla^{h}y^{h}(A^{h})^{-1}(x)) dx\\
&= \liminf\frac{1}{h^{2}} \int_{\Omega^{1}} W(x,(R^{h})^{T}\nabla^{h}y^{h}(A^{h})^{-1}(x)- \mbox{Id}+\mbox{Id}) dx\\
%&= \liminf \int_{\Omega^{1}} \mathcal{Q}_{1}\left(x,\frac{1}{h} \left( R^{h})^{T}\nabla^{h}y^{h}A^{-1}(x)- \mbox{Id} \right) \right)\\
&\geq \frac{1}{2}\int_{\Omega^{1}} {Q}_{3}\Big( x,\overline{A}^{-1}(x')[s(x')+x_{3}\nabla y(x')^{T} \nabla b(x')-\\
&(\overline{A}B) + d(x',x_{3}) \otimes e_{3}]\overline{A}^{-1}(x') \Big)dx\\
&\geq \frac{1}{2}\int_{\Omega^{1}} Q_{2}'(x', t\nabla {y}^{T} \nabla {b}-(\overline{A}B )_{2\times 2}, \overline{A}) dx';
\end{split}
\end{equation}
where we have used Lemma \ref{lemmaweak} applied to $\frac{R^{h}(x')^{T}\nabla^{h} y^{h} \overline{A}(x')-Id}{h}$. 

\end{proof}

\begin{remark}
The same proof would hold if instead of considering $A^{h}(x)=\overline{A}(x')+h(B(x',x_{3})),$ we consider 
\begin{equation}
A^{h}(x)=\overline{A}(x')+hB^{h}(x',x_{3}),
\end{equation}
where $B^{h} \to B$ strongly in $L^{\infty}(\Omega^{1}),$ because
\begin{equation}
(\overline{A}(x')+hB^{h}(x))^{-1}=\overline{A}^{-1}(x')-h\overline{A}^{-1}(x')B^{h}(x)\overline{A}(x')^{-1}+\mathcal{O}(h^{2}),
\end{equation}
therefore
\begin{equation}
\frac{1}{h}\left( (\overline{A}(x')+hB^{h}(x))^{-1}-(\overline{A}^{-1}(x')-h\overline{A}^{-1}(x')B^{h}(x)\overline{A}(x')^{-1}) \right) \to 0
\end{equation}
in $L^{\infty}$ (and in $L^{2}$), then we still have that $S^{h} \rightharpoonup S,$ where $S$ satisfies \eqref{defoverlineS}.
\end{remark}

\subsection{Upper bound (Proposition 2)}

The ansatz is the same as found in \cite{bhattacharya2016plates}, since their proof can be easily adapted to the case $W=W(x, F)$ using hypotheses iv) and v). The ansatz takes the form
\begin{equation}
y^{h}(x',x_{3}) = y(x')+ h x_{3}b(x')+ h^{2}D^{h}(x',x_{3}),
\end{equation}
where 
\begin{equation}
D^{h} (x',x_{3}) = \int_{0}^{x_{3}} d^{h}(x',t)\, dt 
\end{equation}
and $d^{h}$ is an $h-$dependent mollification of $d(x',t),$ where 
\begin{equation}
d(x',t) = \min_{d \in \mathbf{R}^{3}} Q_{3}\left(x',t,\overline{A}^{-1}(x')[t\nabla y^{T} \nabla b-\overline{A}B + d \otimes e_{3}]\overline{A}^{-1}(x') \right).
\end{equation}

We omit the proof and refer the reader to \cite{bhattacharya2016plates} Theorem 3.1.

\section{Proof of Theorem 3 }

From now on, we assume that $\overline{A}(x')=Id,$ which implies ${b}(x')=\nu(x'),$ where $\nu(x')$ is the unit normal. In this case $\nabla y^{T} \nabla \nu$ is the second fundamental form of the surface parametrized by $y,$ we therefore write $\mathbf{II}=\nabla y^{T} \nabla \nu.$ In order to prove theorem 3, we must provide an ansatz whose energy matches that of the lower bound. We will split the proof into three parts: the case $\mathbf{II}=0,$ the case $\mathbf{II}$ bounded away from $0,$ and the general case. The case $\mathbf{II}=0$ will involve a highly oscilatory ansatz, that resembles the Nash-Kuiper embedding. More precisely, it is basically the ansatz used in \cite{friesecke2006hierarchy} with different powers of $h.$ The proof uses recent results about the Monge-Ampere equation \cite{lewicka2017convex}. The case $\mathbf{II}$ bounded away from $0$ involves an essentially different ansatz, in which in-plane and out-of-plane strain combine to relieve the residual metric. Finally, the general case involves combining the two constructions.

First we deal with the case $\mathbf{II}=0.$ The lower bound implies an optimal $s,$ which in general has only $L^{2}$ regularity. We need to approximate it by $C_{0}^{\infty}$ functions. This is done in Lemma \ref{approximationbysmoothfunctions}. Lemma \ref{wrinklesapproximatestrain} is a convex integration-type result which is at the heart of the construction of the ansatz. Lemma \ref{lemmastrong} essentially justifies a Taylor expansion for the energy. Lemma \ref{boundond} is technical: it states that a quantity of interest is uniformly bounded. Building on previous lemmas, Lemma \ref{lastlemma} states that it is possible to construct an ansatz with arbitrary nonlinear strain in the case $\mathbf{II}=0.$ The proof of the upper bound in the case $\mathbf{II}=0$ is a consequence of these lemmas. 
 
From now on, in order to ease notation, we write $Q_{2}(x, F, \text{Id}) = Q_{2}(x, F).$
\begin{lemma}\label{approximationbysmoothfunctions}
For any $H \in L^{2}(\Omega,\mathbf{R}^{2\times 2} )$ we have that 
\begin{equation}
\begin{split}
&\min_{s \in L^{2}(\Omega,\mathbf{R}^{2\times 2}),d \in L^{2}(\Omega^{1}, \mathbf{R}^{3})}\int_{\Omega^{1}} {Q}_{3}\Big( x,[s(x')+x_{3}H(x')-B(x',x_{3}) + d(x',x_{3}) \otimes e_{3}] \Big)dx\\
=&\inf_{s \in C_{0}^{\infty}(\Omega,\mathbf{R}^{2\times 2}),d \in L^{2}(\Omega^{1}, \mathbf{R}^{3})}\int_{\Omega^{1}} {Q}_{3}\Big( x,[s(x')+x_{3}H(x')-B(x',x_{3}) + d(x',x_{3}) \otimes e_{3}] \Big)dx
\end{split}
\end{equation}
\end{lemma}

\begin{proof}
Let $s \in L^{2}(\Omega,\mathbf{R}^{2\times 2})$ and $s_{n} \in C_{0}^{\infty}(\Omega, \mathbf{R}^{2 \times 2})$ such that
\begin{equation}
\parallel s_{n} - s \parallel_{L_{2}} \to 0.
\end{equation}

%Recall that we had previously defined for $X \in \mathbf{R}^{2 \times 2}$ the quadratic form 
%\begin{equation}
%Q_{2}(x',t,X)=\min_{c \in \mathbf{R}^{3}} Q_{3}(x',t,X+c \otimes e_{3}).
%\end{equation}
% Note that the minimizer $c_{min}$ satisfies 
%\begin{equation}
%(L_{3}(x',t)(X+c_{min} \otimes e_{3}))_{3}=0,
%\end{equation}
%where $A_{3}$ denotes the third column of $A.$

 Note that the form $Q_{2}$ is bilinear, and hence there is a tensor $L_{2}$ such that
\begin{equation}
Q_{2}(x',t,X, )=\langle L_{2}(x',t)F,F \rangle.
\end{equation}
Using the symmetry of $L_{3}$, we can use a completing squares argument and deduce (writing $\Psi=x_{3}H(x')-B(x',x_{3}$)
\begin{equation}
\begin{split}
&\int_{\Omega^{1}} {Q}_{2}\Big( x,[s(x')+\Psi] \Big) -\int_{\Omega^{1}} {Q}_{2}\Big( x,[s_{n}(x')+\Psi] \Big)\\
&=\int_{\Omega^{1}} \langle L_{2}(x',t)(s_{n}-s), [s_{n}(x')+s(x')+2\Psi] \rangle dx\\
&\to 0.
\end{split}
\end{equation}
\end{proof}

Next is a lemma which is the technical foundation of the convex-integration construction used in our ansatz.
\begin{lemma}\label{wrinklesapproximatestrain}
Let $h_{n} \to 0$ and let $A \in C^{\infty}(\overline{\Omega}, \mathbf{R}^{2\times 2}_{sym})$ be such that there exists $c \in \mathbf{R}^{+}$ with 
\begin{equation}
A(x) \geq c \mbox{Id}_{2\times 2},
\end{equation}
then there exists $v_{n} \in C^{\infty}(\overline{\Omega})$ and $w_{n} \in C^{\infty}(\overline{\Omega}, \mathbf{R}^{2})$ such that
\begin{equation}\label{cond1}
\parallel A-\left( \frac{1}{2} \nabla v_{n} \otimes \nabla v_{n} + \nabla_{sym} w_{n} \right) \parallel_{C^{0}} \to 0
\end{equation}
and
\begin{equation}\label{cond2}
\begin{split}
\parallel w_{n} \parallel_{C^{0}}  &\to 0 \mbox{ monotonically }\\
\parallel v_{n} \parallel_{C^{0}}  &\to 0 \mbox{ monotonically. }
\end{split}
\end{equation}
Moreover, $v$ and $w$ can be chosen to satisfy 
\begin{equation}\label{cond3}
\limsup \left\{ \parallel \nabla v_{n} \parallel_{C^{0}}+\parallel \nabla w_{n} \parallel_{C^{0}} \right\} < \infty
\end{equation}
and
\begin{equation} \label{cond4}
\parallel h_{n}^{\frac{1}{2}}\nabla^{2}v_{n}\parallel_{C^{0}}+\parallel h_{n}\nabla^{2}w_{n}\parallel_{C^{0}} \to 0.
\end{equation}
\end{lemma}

\begin{proof}
The existence of $v_{n}, w_{n}$ satisfying \eqref{cond1} and \eqref{cond3} follows from proposition 3.2 of \cite{lewicka2017convex}.This reference also shows
\begin{equation}
\begin{split}
\parallel v_{n} \parallel_{C^{0}} &\to 0\\
\parallel w_{n} \parallel_{C^{0}} &\to 0,
\end{split}
\end{equation}
therefore for a subsequence we have that the convergence is monotonic and therefore \eqref{cond2} holds. 

To construct a sequence such that \eqref{cond4} also holds, we start with a sequence satisfying \eqref{cond1}-\eqref{cond3}, and apply a retardation argument as in \cite{padilla2019asymptotic}: define a function $\sigma(n)$ as 
\begin{equation*}
\sigma(1)=1
\end{equation*}
and
\begin{equation*}
\sigma(n+1)=
\begin{cases}
\sigma(n)+1 \mbox{ if } \parallel h_{n}^{\frac{1}{2}}\nabla^{2}v_{\sigma(n)+1}\parallel_{C^{0}}+\parallel h_{n}\nabla^{2}w_{\sigma(n)+1}\parallel_{C^{0}}\leq \frac{1}{\sigma(n)+1}  \\
\sigma(n) \mbox{ if not}.
\end{cases}
\end{equation*}

It is easy to check that $\sigma(n) \to \infty,$ (if not, then $\sigma(n) =k $ for all $n$ big enough, but $h_{n} \to 0,$ therefore there exists $n_{0}$ such that $\parallel h_{n_{0}}^{\frac{1}{2}}\nabla^{2}v_{k+1}\parallel_{C^{0}}+\parallel h_{n_{0}}\nabla^{2}w_{k+1}\parallel_{C^{0}}\leq \frac{1}{k+1}$, therefore $\sigma(n_{0}+1)=k+1,$ contradiction). Then, by definition,
\begin{equation}
\parallel h_{n}^{\frac{1}{2}}\nabla^{2}v_{\sigma(n)+1}\parallel_{C^{0}}+\parallel h_{n}\nabla^{2}w_{\sigma(n)+1}\parallel_{C^{0}} \to 0,
\end{equation}
and \eqref{cond1}-\eqref{cond4} hold with $v_{n}, w_{n}$ replaced by $v_{\sigma(n)}, w_{\sigma(n)}$.
\end{proof}

Before we continue, we need a short lemma, which is the analogue of lemma \ref{lemmaweak} for strong convergence.

\begin{lemma} \label{lemmastrong}
Let $f^{h} \in L^{\infty}(\Omega^{1},  \mathbf{R}^{3 \times 3})$ be such that 
\begin{equation*}
\limsup \frac{1}{h} \parallel f^{h} \parallel_{L^{\infty}} < \infty  \quad \mbox{and} \quad \mbox{sym}\left( \frac{f^{h}}{h} \right) \to \mbox{sym}\left( f \right) \mbox{ in }L^{2},
\end{equation*}
with $f \in C^{\infty}(\Omega^{1}, \mathbf{R}^{3 \times 3})$ then
\begin{equation} \label{eqlemmastrong}
\frac{1}{h^{2}}\lim \int_{\Omega^{1}} W(x,\mbox{Id}+f^{h}(x))dx = \frac{1}{2} \int_{\Omega^{1}} Q_{3}(x,f) dx
\end{equation}
\end{lemma}

\begin{proof}
%First, we assume that
%\begin{equation}
%\parallel \frac{f^{h}}{f} \parallel_{L^{\infty}} \leq k
%\end{equation}
%for some $k.$ Note that there is no loss of generality, since we can consider $k=\parallel f \parallel_{C^{0}}+1,$ then for $g^{h}=f^{h} \mathbf{1}_{}$
The proof is similar to that of  Lemma \ref{lemmaweak}. Let $S^{h}=\{x \in \Omega^{1}| f^{h}(x) \leq {h^{0.9}}\}.$ Proceeding as in Lemma \ref{lemmaweak}, we have 
\begin{equation}
\mathbf{1}_{S^{h}} \to \mathbf{1}_{\Omega^{1}}.
\end{equation}
Proceeding again as in Lemma \ref{lemmaweak}, we have  for $x \in S^{h},$
\begin{equation}
\frac{1}{h^{2}}W(x,\mbox{Id}+f^{h}(x))=\frac{1}{2} Q_{3}(x,\frac{f^{h}}{h}(x))+\textit{o}(1)\left| \frac{f^{h}}{h} \right|^{2}.
\end{equation}

Let $\Xi_{h}=\Omega^{1} \setminus S^{h}.$ Using the fact that the the tangent space to $SO(3)$ at the identity is the space of antisymmetric matrices, we have
\begin{equation}
\mbox{dist}(Id+f^{h}, SO(3))=\left| \mbox{sym}f^{h} \right|+O(1)\left| f^{h} \right|^{2},
\end{equation}
so 
\begin{equation}
\begin{split}
\frac{1}{h^{2}} \int_{\Xi^{h}} W(x,Id+f^{h}(x))dx &\leq \frac{C}{h^{2}} \int_{\Xi^{h}} \mbox{dist}^{2}(Id+f^{h}, SO(3))\\
&\leq \frac{C}{h^{2}} \int_{\Xi^{h}} |\mbox{sym}f^{h}|^{2}+|\mbox{sym}f^{h}||f^{h}|^{2}+|f^{h}|^{4}\\
 &\leq C \int_{\Xi^{h}}\left|\mbox{sym} \left( \frac{f^{h}(x)}{h} \right) \right|^{2} dx+Ch\\
\end{split}
\end{equation}
We also have
\begin{equation}
\begin{split}
\int_{\Omega^{1}} \left|\mbox{sym} \left( \frac{f^{h}(x)}{h} \right) \right|^{2} \mathbf{1}_{\Xi^{h}}dx &= \int_{\Omega^{1}} \left|\mbox{sym} \left( \frac{f^{h}(x)}{h} \right)-\mbox{sym}f(x)+\mbox{sym}f(x)\right|^{2} \mathbf{1}_{\Xi^{h}}dx\\
&= \int_{\Omega^{1}} \left[ \left| \mbox{sym} \left( \frac{f^{h}}{h}-f\right)\right|^{2}+2\langle \mbox{sym}\left( \frac{f^{h}}{h}-f \right),\mbox{sym}f \rangle +\left| \mbox{sym} \left( f \right) \right| ^{2} \right] \mathbf{1}_{\Xi^{h}}dx\\
&\to 0.
\end{split} 
\end{equation}
The first term tends to $0$ by definition, the second by Cauchy-Schwartz, the third one by dominated convergence.

Similarly, we have that
\begin{equation}
\begin{split}
\int_{\Omega^{1}} Q_{3}\left(x,\frac{f^{h}}{h} \right) \mathbf{1}_{\Xi^{h}} dx &\leq k \int_{\Omega^{1}} \left| \mbox{sym}\frac{f^{h}}{h}\right|^{2} \mathbf{1}_{\Xi^{h}} dx \\
&\to 0.
\end{split}
\end{equation}

Hence
\begin{equation}
\begin{split}
\frac{1}{h^{2}} \int_{\Omega^{1}} W(x,\mbox{Id}+f^{h}(x)) dx &= \frac{1}{h^{2}} \left( \int_{S^{h}} W(x,\mbox{Id}+f^{h}(x)) dx+\int_{\Xi^{h}} W(x,\mbox{Id}+f^{h}(x)) dx \right) \\
&= \frac{1}{2} \int_{S^{h}} Q_{3} \left(x,\frac{f^{h}(x)}{h} \right) dx+|S^{h}|o(1)+\frac{1}{h^{2}}\int_{\Xi^{h}} W(x,\mbox{Id}+f^{h}(x)) dx\\
&\to \frac{1}{2} \int_{\Omega^{1}} Q_{3}(x,f(x)) dx.
\end{split}
\end{equation}

\end{proof}

%We need one more lemma to be able to construct the sequence:
%\begin{lemma}
%If $y^{h}:\Omega^{h} \to \mathbf{R}^{3}$ is a minimizing sequence, then
%\begin{equation}
%\lim\frac{1}{h^{2}} \int_{\Omega^{1}} W(x,\nabla^{h}y^{h})dx=\lim  \int_{\Omega^{1}} Q_{3} \left(x, \frac{(\nabla^{h}y^{h})^{T}\nabla^{h}y^{h}-Id}{h} \right)dx.
%\end{equation}
%\end{lemma}

One last observation before writing down the ansatz is that $d(x',x_{3})$ is uniformly bounded if $s(x')$ is uniformly bounded (in particular, if $s(x') \in C^{\infty}(\overline{\Omega})$). This will be necessary in order to bound the error. 

\begin{lemma}\label{boundond}

Let $d(x',t)$ be such that 
\begin{equation}
\begin{split}
&Q_{3}\left(x, [s(x')-B(x',x_{3})+d(x',t)\otimes e_{3}] \right)\\
=&Q_{2}\left(x, \left([s(x')-B_{2\times 2}(x',x_{3})]\right)_{2\times 2} \right).
\end{split}
\end{equation}
Then $d \in L^{\infty}(\Omega^{1}, \mathbf{R}^{3})$ and we have the pointwise bound
\begin{equation}\label{Linftyboundond}
| d(x',x_{3}) | \leq \left( \sqrt{\frac{C}{c}} + 1\right)  \parallel s-B \parallel_{L^{\infty}} ,
\end{equation}
where $c,C$ are such that
\begin{equation}
c \parallel F \parallel^{2} \leq Q_{3}(x,F) \leq C \parallel F \parallel^{2} 
\end{equation}
for any symmetric $F$, see properties iv), v) of $W$ .
\end{lemma}

\begin{proof}

 To deduce equation \eqref{Linftyboundond}, note that we can write
\begin{equation}
\begin{split}
Q_{3}\left( x, s-B+d\otimes e_{3} \right) &\geq c \left| s(x')-B(x',x_{3})+\mbox{sym}\left(d(x',x_{3})\otimes e_{3} \right) \right|^{2}\\
&\geq c \left( |d(x',x_{3})| - \parallel s- B \parallel_{L^{\infty}} \right)^{2}.
\end{split}
\end{equation}
We also have
\begin{equation}
\begin{split}
Q_{2}\left( x, [s-B]_{2\times 2} \right) &\leq Q_{3}\left( x, s-B \right)\\
& \leq C \parallel s-B \parallel_{L^{\infty}}^{2}.
\end{split}
\end{equation}
Combining these two estimates yields \eqref{Linftyboundond}.

\end{proof}

Since we require that the ansatz is in $W^{1,2},$ we cannot exactly plug in $d(x',x_{3})$ since in general it has only $L^{2}$ regularity. Instead, we need a suitable smooth approximation: let $d^{h}(x',t)$ be a sequence of $C^{\infty}$ functions that converge to $d(x',t),$ strongly in $L^{2}$ and such that $h\nabla' d^{h}(x',t)$ converges strongly to $0$ in $L^{\infty}.$ For example, take the trivial extension of $d$ to $\mathbf{R}^{3},$ and take 
\begin{equation}\label{defofdh}
d^{h}=d \ast \mu_{h^{\frac{1}{2}}},
\end{equation}
then by Young's inequality $\parallel h\nabla' d^{h}(x',t) \parallel_{L^{\infty}} \leq k h^{\frac{1}{2}} \to 0.$ Let $D^{h}(x',t)=\int_{0}^{t} d^{h}(x',s) ds.$

We are finally ready to write down the ansatz that achieves the lower bound: the ansatz is $y^{h}$ defined as 
\begin{equation}\label{ansatz}
\begin{split}
y^{h}(x',x_{3})=&(1-hC)y(x')+h[x_{3}(1-Ch)\nu(x')+Q(x')\begin{bmatrix}
w_{1}^{h}\\
w_{2}^{h}\\
0
\end{bmatrix}
+h^{2}QD^{h}(x',x_{3})\\
&+h^{\frac{1}{2}} Q\begin{bmatrix}
0\\
0\\ 
v^{h}
\end{bmatrix}-h^{\frac{3}{2}} x_{3} Q\nabla v^{h}-\frac{1}{2}h^{2}x_{3}Q \begin{bmatrix}
0\\
0\\
|\nabla v^{h}|^{2}
\end{bmatrix}
\end{split}
\end{equation}
where $Q=[\nabla y | {\nu}]$ and $v^{h}, w^{h}, C$ will be chosen later. An explicit calculation shows that
% and 
%\begin{equation}
%d^{h}(x',t)=\int_{0}^{t} d^{h}(x',x_{3})
%\end{equation}
%and $d^{h}(x',x_{3})\in L^{2}(\Omega^{1}, \mathbf{R}^{3})$ is such that
%\begin{equation}
%Q_{3}\left( x, x_{3} \mathbf{II}+B(x',x_{3})+s(x')+d^{h}(x',x_{3})\otimes e_{3} \right) =Q_{2}\left( x, x_{3} \mathbf{II}+\left( B(x',x_{3} \right)_{2\times 2} +s(x') \right)
%\end{equation}
\begin{equation}
\begin{split}
\nabla^{h}y^{h}=&(1-Ch)Q+h[Q\nabla'w^{h}+Qd^{h}(x',x_{3}) \otimes e_{3}]+h^{2}\nabla'QD^{h}(x',x_{3})\\
&+h^{\frac{1}{2}}Q\begin{bmatrix}
0_{2\times 2}& -\nabla (v^{h})^{T}\\
\nabla v^{h} & 0
\end{bmatrix}
-h^{\frac{3}{2}} x_{3}Q\nabla^{2}v^{h}\\
&-\frac{1}{2}hQ\begin{bmatrix}
0\\
0\\
|\nabla v^{h}|^{2}
\end{bmatrix} \otimes e_{3}-\frac{1}{2}h^{2}x_{3}Q\begin{bmatrix}
0&0&0\\
0&0&0\\
\partial_{x_{1}}|\nabla v^{h}|^{2}&\partial_{x_{2}}|\nabla v^{h}|^{2}&0
\end{bmatrix}.
\end{split}
\end{equation}
Recall that by hypothesis $\mathbf{II}=0,$ which means $y(x')$ is a plane. We also have
\begin{equation}
\begin{split}
(\nabla^{h}y^{h})^{T}\nabla^{h}y^{h}=&(1-2Ch)Id+2h[\nabla_{sym}(w^{h})+\mbox{sym}(d^{h}(x',x_{3})\otimes e_{3})]\\
&+h\begin{bmatrix}
\nabla v^{h} \otimes \nabla v^{h}&0\\
0& 0
\end{bmatrix} +\mathcal{O}(h^{\frac{3}{2}})(1+\parallel \nabla^{2}v^{h} \parallel_{L^{\infty}}).
\end{split}
\end{equation}
Here, to bound the error, we have used that 
\begin{equation}
\parallel Q\nabla'w^{h} \parallel_{L^{\infty}}+\parallel Qd^{h}(x',x_{3}) \parallel_{L^{\infty}}+\parallel h^{\frac{1}{2}}\nabla'QD^{h}(x',x_{3}) \parallel_{L^{\infty}}+\parallel \nabla v^{h} \parallel_{L^{\infty}} \leq C.
\end{equation}
We have also used that 
\begin{equation}
\begin{bmatrix}
0_{2\times 2}& -\nabla (v^{h})^{T}\\
\nabla v^{h} & 0
\end{bmatrix}^{T}
\begin{bmatrix}
0_{2\times 2}& -\nabla (v^{h})^{T}\\
\nabla v^{h} & 0
\end{bmatrix}
-\begin{bmatrix}
0\\
0\\
|\nabla v^{h}|^{2}
\end{bmatrix} 
=
\begin{bmatrix}
\nabla v^{h} \otimes \nabla v^{h}&0\\
0& 0
\end{bmatrix}.
\end{equation}

%\begin{equation}\label{ansatz}
%\begin{split}
%y^{h}(x',x_{3})=&(1-hC)y(x')+h[x_{3}(1-Ch)\nu+Qw^{h}(x')]+h^{2}QD(x',x_{3})\\
%&+h^{\frac{1}{2}} Q\begin{bmatrix}
%0\\
%0\\ 
%v^{h}
%\end{bmatrix}-h^{\frac{3}{2}} x_{3} Q\nabla v^{h}-h^{2}x_{3}Q \begin{bmatrix}
%0\\
%0\\
%|\nabla v^{h}|^{2}
%\end{bmatrix}
%\end{split}
%\end{equation}
%where $Q=[\nabla y | \nu]$ and 
%\begin{equation}
%D(x',t)=\int_{0}^{t} d(x',x_{3})
%\end{equation}
%and $d(x',x_{3})\in L^{2}(\Omega^{1}, \mathbf{R}^{3})$ is such that
%\begin{equation}
%Q_{3}\left( x, B(x',x_{3})+s(x')+d(x',x_{3})\otimes e_{3} \right) =Q_{2}\left( x, \left( B(x',x_{3} \right)_{2\times 2} +s(x') \right).
%\end{equation}
%We then get that
%\begin{equation}
%\begin{split}
%\nabla^{h}y^{h}=&(1-Ch)Q+h[Q\nabla'w^{h}+Qd(x',x_{3}) \otimes e_{3}]+h^{2}\nabla'\left[QD(x',x_{3})\right]\\
%&+h^{\frac{1}{2}}Q\begin{bmatrix}
%0_{2\times 2}& -\nabla (v^{h})^{T}\\
%\nabla v^{h} & 0
%\end{bmatrix}
%+h^{\frac{3}{2}} x_{3}\nabla^{2}Qv^{h}\\
%&-hQ\begin{bmatrix}
%0\\
%0\\
%|\nabla v^{h}|^{2}
%\end{bmatrix} \otimes e_{3}-h^{2}x_{3}Q\begin{bmatrix}
%0&0\\
%0&0\\
%\partial_{x_{1}}|\nabla v^{h}|^{2}&\partial_{x_{2}}|\nabla v^{h}|^{2}
%\end{bmatrix},
%\end{split}
%\end{equation}
%and
%\begin{equation}
%\begin{split}
%(\nabla^{h}y^{h})^{T}\nabla^{h}y^{h}=&(1-Ch)Id+h[\nabla_{sym}(w^{h})+\mbox{sym}(d(x',x_{3})\otimes e_{3})]+h^{2}\nabla'D(x',x_{3})\\
%&+h\begin{bmatrix}
%\nabla v^{h} \otimes \nabla v^{h}&0\\
%0& 0
%\end{bmatrix} +\mathcal{O}(h^{\frac{3}{2}})(1+\parallel \nabla^{2}v^{h} \parallel).
%\end{split}
%\end{equation}

We need one more lemma to conclude:

\begin{lemma}\label{lastlemma}
For any $s \in C^{\infty}(\Omega, \mathbf{R}^{2\times 2}),$ and $d\in L^{\infty}(\Omega_{1} \to \mathbf{R}^{3})$ there exists $y^{h}$ given by \eqref{ansatz}, such that 
\begin{equation}
\mbox{sym}\left( \overline{S}(x',x_{3})\right) = \mbox{sym}\left(s(x')+d(x',x_{3})\otimes e_{3}\right),
\end{equation}
$ \limsup \parallel S^{h} \parallel_{L^{\infty}} <\infty $, and $\mbox{sym}(\overline{S}^{h}) \to \ \mbox{sym}(\overline{S})$ strongly in $L^{2}$, where $\overline{S}^{h}(x',x_{3}), S^{h}(x',x_{3}), \overline{S}(x',x_{3})$ and $S(x',x_{3})$ are given by \eqref{defoverlineSh}- \eqref{defS}, with $\overline{A}=Id$ and $R^{h}$ given by \eqref{existRhfields}.
\end{lemma}

\begin{proof}
Let $v^{h},w^{h}$ and $C$ be as in lemma \ref{wrinklesapproximatestrain} for $A=s(x')+CId.$ First, note that
\begin{equation}
\mbox{dist}(\nabla^{h}y^{h}, SO(3)) \leq C h,
\end{equation}
uniformly in $x$ and $h$ since the order $h$ terms are uniformly bounded in $x$ and $h$, and the term of order $h^{\frac{1}{2}}$ is skew symmetric and uniformly bounded in $x$ and $h.$ Hence  
\begin{equation}
 \limsup \parallel S^{h} \parallel_{L^{\infty}} < \infty.
\end{equation}
Note that from the definition of $\overline{S}^{h},$ we have that
\begin{equation}
(\nabla^{h}y^{h})^{T}(\nabla^{h}y^{h})=Id+h((\overline{S}^{h})^{T}+\overline{S}^{h})+h^{2}(\overline{S}^{h})^{T}\overline{S}^{h}.
\end{equation}
%Let $v^{h}, w^{h}$ be as in lemma \ref{wrinklesapproximatestrain} with $A$ replaced by $s+C\, \mbox{Id}.$

 From the form of $y^{h}$ (equation \ref{ansatz}) we have that
\begin{equation}
\begin{split}
\frac{\sqrt{(\nabla^{h}y^{h})^{T}(\nabla^{h}y^{h})}-Id}{h}&=\frac{1}{2}\nabla v_{n} \otimes \nabla v_{n} + \nabla_{sym} w_{n} + \mbox{sym}(d^{h}\otimes e_{3})+{o}(1)\\
&\to \mbox{sym}\left(s(x')\right) + \mbox{sym}(d\otimes e_{3})
\end{split}
\end{equation}
strongly in $L^{2}$. By Taylor expanding $\sqrt{}$ at the identity, we get (recall that by hypothesis $\limsup \parallel S^{h} \parallel < \infty $)
\begin{equation}
\|\frac{\sqrt{Id+h((\overline{S}^{h})^{T}+\overline{S}^{h})+h^{2}(\overline{S}^{h})^{T}\overline{S}^{h}}-Id}{h}-\frac{(\overline{S}^{h})^{T}+\overline{S}^{h}}{2} \|{L^{2}} \to 0,
\end{equation}
hence
\begin{equation}
\left( \frac{(\overline{S}^{h})^{T}+\overline{S}^{h}}{2}\right)_{2\times 2} \to \mbox{sym}\left( s(x') \right).
\end{equation}

%Let $d(x',t)$ be such that 
%\begin{equation}
%\begin{split}
%&Q_{3}\left(x, [s(x')-B(x',x_{3})+d(x',t)\otimes e_{3}] \right)\\
%=&Q_{2}\left(x, \left([s(x')-B(x',x_{3})]\right)_{2\times 2} \right)
%\end{split}
%\end{equation}
%For the last row and column, let $d^{h}(x',t)$ be a sequence of $C^{\infty}$ functions that converge to $d(x',t),$ strongly in $L^{2}$ and such that $h\nabla' d^{h}(x',t)$ converges strongly to $0$ in $L^{\infty}.$ For example, take the trivial extension of $d$ to $\mathbf{R}^{3},$ and take $d^{h}=d \ast \mu_{h^{\frac{1}{2}}},$ then by Young's inequality $\parallel h\nabla' d^{h}(x',t) \parallel_{L^{\infty}} \leq k h^{\frac{1}{2}} \to 0.$ Let $D^{h}(x',t)=\int_{0}^{t} d^{h}(x',s) ds.$

 By defining $y^{h}$ this way, we have that 
\begin{equation}
\frac{(\overline{S}^{h})^{T}+\overline{S}^{h}}{2} \to \mbox{sym}\left( s(x') \right)+\mbox{sym}(d(x',t)\otimes e_{3}).
\end{equation}
%since $\mbox{sym}\nabla y^{T}\nabla \nu=0.$
\end{proof}

Now we turn to prove Theorem \ref{theorem2} in the case $\mathbf{II}=0.$

\begin{proof}

(Of Theorem \ref{theorem2} in the case $\mathbf{II}=0.$)

Let $s(x') \in C^{\infty}(\Omega, \mathbf{R}^{2\times 2}).$ Let $y^{h}$ be given by \eqref{ansatz}, with $v^{h}$ and $w^{h}$ satisfing \eqref{cond1}, \eqref{cond2}, \eqref{cond3}, \eqref{cond4} with $A=s+ C \ \mbox{Id}$ (existence of such a sequence in guaranteed by lemma \ref{lastlemma}). Let $d(x',x_{3}) \in {L^{2}}(\Omega^{1}, \mathbf{R}^{3})$ be such that
\begin{equation}
Q_{3}\left(x,[s(x') + d(x',x_{3})\otimes e_{3}-B(x)] \right)=Q_{2}\left(x, s(x')- B_{2\times 2}(x) \right).
\end{equation}

Recall that we had previously defined $R^{h}(x') \in SO(3)$ such that
\begin{equation}
 \int_{\Omega^{1}} \left| \nabla^{h}y^{h}-R^{h}(x')  \right|^{2} dx \leq Ch^{2},
\end{equation}
and the quantity $\overline{S}^{h}(x',x_{3})$ as
\begin{equation}
\overline{S}^{h}(x',x_{3})=\frac{1}{h} \left( R^{h}(x')^{T}\nabla^{h}y^{h} -\mbox{Id} \right).
\end{equation}

Using Lemmas \ref{lemmastrong} and \ref{lastlemma} we now conclude:

\begin{equation}
\begin{split}
\lim_{h \to 0}\frac{1}{h^{2}} \mathcal{E}^{h}(y^{h})&= \lim_{h \to 0}\frac{1}{h^{2}}\int_{\Omega^{1}}W(x,\nabla^{h} y^{h}\left( A^{h}\right)^{-1}(x))\\
&=\lim_{h \to 0}\frac{1}{h^{2}}  \int_{\Omega^{1}}W(x,(R^{h})^{T}\nabla^{h} y^{h}\left( A^{h} \right) ^{-1}(x))\\
&=\lim_{h \to 0}\frac{1}{h^{2}}  \int_{\Omega^{1}}W(x,(R^{h})^{T}\nabla^{h} y^{h}\left( A^{h} \right) ^{-1}(x)-Id+Id)\\
%&= \lim_{h \to 0}\frac{1}{h^{2}} \int_{\Omega^{1}}W(x,Id+hS^{h}(x',x_{3}))\\
&= \frac{1}{2} \int_{\Omega^{1}}Q_{2}(x,[s(x')-B_{2\times 2}(x)] ).
\end{split}
\end{equation}

It is tempting to use this construction to build an ansatz directly: take 
\begin{equation}\label{defs*}
s^*(x')=\mbox{argmin}_{s} \int_{\Omega^{1}}Q_{2}(x,[s-B_{2\times 2}(x)] ) \, dt
\end{equation}
then construct $v^{h}, w^{h}$ as in lemma \ref{wrinklesapproximatestrain} with $A=s^*+C\, \mbox{Id}$ and conclude. The problem is that these arguments would only work if $s^*$ were known to be smooth. Instead, we approximate the minimizer by smooth functions. 

Let $s^*$ be given by equation \eqref{defs*}
%\begin{equation}
%s^*=\mbox{argmin}_{s(x')} \int_{\Omega^{1}}Q_{2}(x,[s(x')-B_{2\times 2}(x)] )
%\end{equation}
and $s_{n} \in C_{0}^{\infty}(\overline{\Omega}, \mathbf{R}^{2\times 2})$ be such that
\begin{equation}
\parallel s^*-s_{n} \parallel_{L^{2}} \to 0.
\end{equation}

The previous argument shows that there exists $y^{h_{n}}$ with $h_{n} \leq \frac{1}{n}$ such that 
\begin{equation}
\left| \frac{1}{h_{n}^{2}} \mathcal{E}^{h_{n}}(y^{h_{n}}) -  \frac{1}{2}\int_{\Omega^{1}}Q_{2}(x,[s_{n}(x')-B_{2\times 2}(x)] ) \right| \leq \frac{1}{n},
\end{equation}
therefore, using lemma \ref{approximationbysmoothfunctions} we have
\begin{equation}
\begin{split}
\frac{1}{h^{2}} \lim_{h \to 0}\mathcal{E}^{h}(y^{h})&= \frac{1}{2}\int_{\Omega^{1}}Q_{2}(x,[s^*(x')-B_{2\times 2}(x)] )\\
&=\mathcal{I}(y).
\end{split}
\end{equation}

\end{proof}

For the case $\mathbf{II} \neq 0,$ we will use two levels of approximation: first, given an arbitrary $W^{2,2}$ isometric immersion, we will approximate it by a smooth and \textit{nice} isometric immersion. As in the previous step, the lower bound implies an optimal $s,$ which in general has only $L^{2}$ regularity and we approximate it by $C_{0}^{\infty}$ functions. Unlike the case of $\mathbf{II}=0$, we will not prove intermediate lemmas. 

\begin{proof}

(Of Theorem \ref{theorem2}, case $\mathbf{II}$ everywhere nonzero)

 We borrow notation from \cite{schmidt2007plate}. We let $\mathcal{A}_{0}$ be the set of smooth ismetric immersions that allow for a partition into finitely many bodies and arms (see \cite{schmidt2007plate} for a definition). Using \cite{hornung2011approximation} and \cite{hornung2011fine}, we have (as noted after Theorem 2.4 of \cite{schmidt2007plate}) that $\mathcal{A}_{0}$ is strongly $W^{2,2}$ dense in the space of $W^{2,2}$ isometric immersions \footnote{By using the abovementioned references, we can drop the requirement that $\Omega$ is convex, which was present in \cite{schmidt2007plate}, and require only mild regularity on the boundary. It is only for this step that we need $\partial \Omega$ to be piecewise $C^{1},$ although the theorem is still true under slightly weaker assumptions, see \cite{hornung2011approximation} and \cite{hornung2011fine} for more details.}. We will henceforth assume $y  \in \mathcal{A}_{0}.$ Our results can be extended to a general isometric immersion $y \in W^{2,2}$ by density.

 We can use Lemma 3.3 of  \cite{schmidt2007plate}, which states that for any $ y \in \mathcal{A}_{0}$ and any $s(x') \in C^{\infty}(\Omega, \mathbf{R}^{2\times 2}_{sym})$ such that $s$ vanishes in a neighborhood of $\mathbf{II}=0,$ there exists $g \in C^{\infty}(\overline{\Omega}, \mathbf{R^{2}})$ and $\alpha \in C^{\infty}(\overline{\Omega})$ such that
\begin{equation}\label{Schmidt's equation}
s(x')=\nabla_{sym} g + \alpha \mathbf{II}.
\end{equation}
To get started, we present an ansatz that works when $\mathbf{II}\neq 0$ \footnote{This assumption, in combination with the previous regularity assumptions, means that we can take $|\mathbf{II}|\geq c > 0.$}, so that any smooth $s(x')$ has a representation of the form \eqref{Schmidt's equation}. These ideas can be found in \cite{schmidt2007plate}, but our presentation will be different. Let
\begin{equation}\label{ansatz2}
y^{h}=y(x')+h[x_{3}\nu(x')+Qg(x')]+h^{2}QD^{h}(x',x_{3}),
\end{equation}
where now $g: {\Omega} \to \mathbf{R}^{3}$ and $Q=[\nabla'y,\nu]$ (note that $Q$ is no longer constant). Let $g=(g',g_{3}).$ Using \eqref{Schmidt's equation} we can choose $g$ such that
\begin{equation}
\mbox{sym}(s^*(x'))=\nabla_{sym}g'+g_{3}\mathbf{II},
\end{equation}
where $s^*$ is given by equation \eqref{defs*}.
%\begin{equation}
%s^*=\mbox{argmin}_{s(x')} \int_{\Omega^{1}}Q_{2}(x,[s(x')-B_{2\times 2}(x)] ).
%\end{equation}

Using an approximation argument as before, we can assume that $s^{*}$ is in $C^{\infty}(\overline{\Omega}, \mathbf{R}^{2\times 2}_{sym}).$ The vector $D^{h}(x',x_{3})$ in \eqref{ansatz2} plays the same role as the analogous term in our previous ansatz: it satisfies $\partial_{3} D^{h}(x',x_{3})=d^{h}(x',x_{3}),$ where the vector $d^{h}$ is a smoothed version of $d,$ and the vector $d(x',x_{3}) \in {L^{\infty}}(\Omega^{1}, \mathbf{R}^{3})$ is such that
\begin{equation}
\begin{split}
&Q_{3} \left(x, x_{3}\nabla y^{T}\nabla \nu+\nabla'g+\begin{bmatrix}
g_{3} \mathbf{II}\\
-g' \nabla y^{T}\nabla \nu
\end{bmatrix}
+d(x',x_{3})\otimes e_{3} -B(x) \right)\\
=&Q_{2} \left(x,s(x')+x_{3}\nabla y^{T}\nabla \nu -\left( B(x)\right)_{2\times 2} \right).
\end{split}
\end{equation}
The relationship between $d^{h}(x',x_{3})$ and $d(x',x_{3})$ was discussed in our treatment of ansatz \eqref{ansatz}, and the same arguments apply here. Recall our convention for summing matrices of different dimensions: the smaller matrix is viewed as the top left block of the bigger matrix. We now compute the gradient of the ansatz: 
\begin{multline}
\nabla^{h}y^{h}=\\
Q+hQ\left( x_{3}\nabla y^{T}\nabla \nu+\nabla'g+\begin{bmatrix}
g_{3} \nabla y^{T}\nabla \nu \\
- g' \nabla y^{T}\nabla \nu
\end{bmatrix}
+d^{h}(x',x_{3})\otimes e_{3} \right)+h^{2}\nabla'\left( QD^{h}(x',x_{3})\right).
\end{multline}
A more detailed version of this computation can be found in \cite{schmidt2007plate}, but essentially it uses the fact that 
\begin{equation}
Q^{T}\nabla \left( Qg \right)= \nabla' g + \begin{bmatrix}
g_{3} \nabla y^{T}\nabla \nu\\
-g' \nabla y^{T}\nabla \nu
\end{bmatrix}.
\end{equation}
This can be verified with an explicit computation, and using the following facts
\begin{itemize}
\item[$\bullet$] $\partial_{i} y \cdot \partial_{j} y = \delta_{ij}$ 
\item[$\bullet$] $\partial_{i} y \cdot \nu = 0$
\item[$\bullet$] $\partial_{i} y \cdot \partial_{jk} y = 0$.
\end{itemize}

%We proceed similarly: let $M(x') \in C^{\infty}(\Omega, \mathbf{R}^{2\times 2}).$ Let $y^{h}$ be given by \eqref{ansatz2}, with $(g',g_{3})$ such that
%\begin{equation}
%M(x')=\nabla_{sym}g'+g_{3}\mathbf{II}
%\end{equation}

It then follows that
\begin{equation}
\begin{split}
& \lim_{h \to 0}\frac{1}{h^{2}}\mathcal{E}^{h}(y^{h})\\
=& \lim_{h \to 0}\frac{1}{h^{2}}\int_{\Omega^{1}}W(x,\nabla^{h} y^{h}\left(A^{h} \right)^{-1}(x))\\
=&\lim_{h \to 0}\frac{1}{h^{2}}  \int_{\Omega^{1}}W(x,Q^{T}\nabla^{h} y^{h}\left(A^{h} \right)^{-1}(x))\\
=&\lim_{h \to 0}\frac{1}{h^{2}}  \int_{\Omega^{1}}W\left( x,Id+h\left( x_{3}\nabla y^{T}\nabla \nu+\nabla'g+\begin{bmatrix}
g_{3} \mathbf{II}\\
-g' \mathbf{II}
\end{bmatrix}
+d^{h}(x',x_{3})\otimes e_{3} \right)+h^{2}\nabla'D^{h}(x',x_{3}) - hB \right)\\
=& \frac{1}{2} \int_{\Omega^{1}}Q_{2}(x,[x_{3}\nabla y^{T} \nabla \nu+s(x')-B_{2\times 2}(x)] )
\end{split}
\end{equation}

As in the case $\mathbf{II}=0,$ we cannot simply take $s(x')$ to be the minimizer of
\begin{equation}
\int_{\Omega} Q_{2}\left( x, x_{3} \nabla y^{T} \nabla\nu +s(x')-B_{2\times 2} \right) dt \, dx'.
\end{equation}
But by lemma \ref{approximationbysmoothfunctions} we can approximate the minimum using a sequence of smooth functions $s(x')$ and let $y^{h}=y^{h_{j}}$ be obtained using the ansatz associated with $s_{j}$ (with $h$ sufficiently small). This suffices to establish the upper bound when $\mathbf{II}$ is everywhere nonzero.

\end{proof}

Finally, we turn to the general case, where $\mathbf{II}$ is neither identically vanishing, nor everywhere nonzero. The strategy will be to glue the two previous ansatzes with a transition layer at the boundary of $\{ \mathbf{II}=0 \}.$ The main challenge is to build the transition in such a way that the transition layer has negligible energy. We will prove one intermediate lemma, Lemma \ref{second_approx}, in order to justify taking $s$ to be a particular type of smooth function. 

For the general case, we introduce a further approximation: apart from approximating $s$ by $C^{\infty}_{0}$ functions, and $y$ by nice smooth isometric immersions, we will approximate the sets where $\mathbf{II}$ is $0$ and nonzero. Recall that we are assuming that $h$ is a fixed sequence such that $h_{n} \to 0.$ We again assume $y \in \mathcal{A}_{0}.$ 

We assume that the sets  $\{ \mathbf{II}=0\} $ and $\{ \mathbf{II} \neq 0\}$ are nonempty sets, and we introduce the notation 
\begin{equation}
F=\partial \{x \in \Omega| \mathbf{II}=0\},
\end{equation} 
and  
\begin{equation}
F_{\epsilon}= \bigcup_{x\in F} B(x,\epsilon)
\end{equation} 
for the boundary(which we can assume has measure $0$\footnote{More precisely, we may assume that the set where $\mathbf{II}=0$ has finitely many connected components by Proposition 5 of \cite{hornung2011approximation}, and the boundary of each connected component has measure $0$ by Lemma 1 of \cite{hornung2011fine}}), and a thickened boundary. Let 
\begin{equation}
\Omega^ {+}=\{x \in \Omega| \mathbf{II}\neq 0\}
\end{equation}
 and 
\begin{equation}
\Omega^{0}=\{x \in \Omega| \mathbf{II}=0\}.
\end{equation}
Let 
\begin{equation}
\Omega_{+}^{\epsilon}=\Omega^{+} \setminus F_{\epsilon}
\end{equation}
and 
\begin{equation}
\Omega_{0}^{\epsilon}=\Omega^{0} \setminus F_{\epsilon}
\end{equation}
be the points in $\Omega^{+}$ (respectively $\Omega^{0}$) away from the thickened boundary layer. 

To succesfully combine the two ansatzes, it is important that they transition smoothly along a boundary layer. This is the purpose of the next lemma. 

\begin{lemma}\label{second_approx}
Let $K \in \mathbf{R}^{2\times 2}$ and let $C^{*}_{K}(\Omega, \mathbf{R}^{2\times 2})=\{ f \in C^{\infty}{(\Omega, \mathbf{R}^{2\times 2})} | f(x)=K \mbox{ for } x\in F_{\epsilon} \mbox{ for some } \epsilon\}.$ For all $K \in \mathbf{R}^{2\times 2}$ we have
\begin{equation}
\begin{split}
&\min_{s \in L^{2}(\Omega,\mathbf{R}^{2\times 2}),d \in L^{2}(\Omega^{1}, \mathbf{R}^{3})}\int_{\Omega^{1}} {Q}_{3}\Big( x,[s(x')+x_{3}H(x')-B(x',x_{3}) + d(x',x_{3}) \otimes e_{3}] \Big)dx\\
=&\inf_{s \in C^{*}_{K}(\Omega,\mathbf{R}^{2\times 2}),d \in L^{2}(\Omega^{1}, \mathbf{R}^{3})}\int_{\Omega^{1}} {Q}_{3}\Big( x,[s(x')+x_{3}H(x')-B(x',x_{3}) + d(x',x_{3}) \otimes e_{3}] \Big)dx.
\end{split}
\end{equation}
\end{lemma}

\begin{proof}
The main task is to show that $C^{*}_{K}(\Omega,\mathbf{R}^{2\times 2})$ is dense in $L^{2}(\Omega,\mathbf{R}^{2\times 2}).$ For this, let $f \in L^{2}(\Omega,\mathbf{R}^{2\times 2}),$ and consider, for $\delta$ to be determined later, 
\begin{equation}
f_{\epsilon, \delta}=(\overline{f \mathbf{1}_{\Omega\setminus F_{\epsilon}}+K \mathbf{1_{F_{\epsilon}}}}) \ast \mu_{\delta},
\end{equation}
where $\overline{g}$ denotes the trivial extension of $g$ to $\mathbf{R}^{n}.$ 

%By choosing $\epsilon$ and $\delta$ appropriately, we have that 
%\begin{equation}
%f_{\epsilon, \delta} \in C^{*}_{K}(\Omega,\mathbf{R}^{2\times 2}) \quad f_{\epsilon, \delta} \to f \mbox{ in }L^{2}.
%\end{equation}

Notice that if 
\begin{equation}
\delta<\frac{\epsilon}{2},
\end{equation}
then
\begin{equation}
f_{\epsilon, \delta}= K
\end{equation}
in $F_{\frac{\epsilon}{4}}$  and therefore $f_{\epsilon, \delta} \in C^{*}_{K}$. We also have that
\begin{equation}
\parallel f_{\epsilon, \delta(\epsilon)} - f \parallel_{L^{2}} \to 0,
\end{equation}
since
\begin{equation}
\begin{split}
\parallel f-f \mathbf{1}_{ \Omega \setminus F_{\epsilon}} \parallel_{L{2}}&=\parallel f \mathbf{1}_{F_{\epsilon}}\parallel_{L{2}} \to 0, \\
\parallel \left( f \mathbf{1}_{\Omega\setminus F_{\epsilon}} \right) - \left( f \mathbf{1}_{\Omega \setminus F_{\epsilon}} - K \mathbf{1}_{F_{\epsilon}} \right)  \parallel_{L_{2}}&=\parallel K \mathbf{1}_{F_{\epsilon}} \parallel_{L_{2}} \to 0, \\
 \parallel f \mathbf{1}_{\Omega \setminus F_{\epsilon}} - K \mathbf{1}_{F_{\epsilon}} - (\overline{f \mathbf{1}_{\Omega\setminus F_{\epsilon}}+K \mathbf{1_{F_{\epsilon}}}}) \ast \mu_{\delta} \parallel_{L_{2}}& \to 0, 
\end{split}
\end{equation}
(in the last line, we assume $\epsilon$ is fixed and $\delta \to 0.$) Here we have used that $|F_{\epsilon}| \to 0,$ since $F$ is closed and has measure $0$. Hence to prove convergence, given an error $\eta,$ we can choose $\delta$ such that $\delta \leq \frac{\epsilon}{2}$ and each error is smaller than $\frac{\eta}{3}.$ The lemma follows easily from this approximation result by an argument we used in lemma \ref{approximationbysmoothfunctions}.
\end{proof}

We now turn to the proof of Theorem \ref{theorem2} in the general case.

\begin{proof}

(Of Theorem \ref{theorem2}, general case).

Both our ansatzes began by considering an arbitrary $s \in C^{\infty}(\overline{\Omega}, \mathbf{R}^{2\times 2}_{sym})$, so of course we will do the same here: let $s \in C^{\infty}(\overline{\Omega}, \mathbf{R}^{2\times 2}_{sym}),$ and let $C$ be such that 
\begin{equation}
\widetilde{s}=s+CId > c Id
\end{equation}
for some $c \in \mathbf{R}^{+}.$ Let $v^{h}, w^{h}$ be such that \eqref{cond1}-\eqref{cond4} hold for $A=\widetilde{s}.$ Let $\phi_{\epsilon} \in C^{\infty}(\overline{\Omega})$ be such that
\begin{equation}
\begin{split}
\phi_{\epsilon} &=1 \mbox{ in } \Omega_{0}^{\epsilon}\\
\phi_{\epsilon} &=0 \mbox{ in } \Omega_{+},
\end{split}
\end{equation}  
(for example, take $\phi_{\epsilon}=\mathbf{1}_{\Omega_{0}^{\epsilon}} \ast \mu_{\epsilon}$), let $v^{h}_{\epsilon}=v^{h} \phi_{\epsilon}$ and $w^{h}_{\epsilon}=w^{h} \phi_{\epsilon}.$ Let $\eta_{\epsilon} \in C^{\infty} (\overline{\Omega})$ be such that
\begin{equation}
\begin{split}
\eta_{\epsilon} &=1 \mbox{ in } \Omega_{+}^{2\epsilon}\\
\eta_{\epsilon} &=0 \mbox{ in } \Omega \setminus \Omega_{+}^{\epsilon},
\end{split}
\end{equation} 
for example, take $\eta_{\epsilon}=\mathbf{1}_{\Omega_{+}^{2\epsilon}} \ast \mu_{\epsilon}.$

By defining $\phi_{\epsilon}$ and $\eta_{\epsilon}$ this way, we get the bound
\begin{equation}
\begin{split}
\parallel \nabla \eta_{\epsilon} \parallel_{L^{\infty}}+\parallel \nabla \phi_{\epsilon} \parallel_{L^{\infty}} &\leq \frac{k}{\epsilon}\\
\parallel \nabla \nabla \eta_{\epsilon} \parallel_{L^{\infty}}+\parallel \nabla \nabla \phi_{\epsilon} \parallel_{L^{\infty}} &\leq \frac{k}{\epsilon^{2}}
\end{split}
\end{equation}
for some $k.$

Let $(g'_{\epsilon},g^{3}_{\epsilon})$ be such that \eqref{Schmidt's equation} holds with  $g=g'_{\epsilon}$, $\alpha=g^{3}_{\epsilon}$ and 
\begin{equation}
s=\widetilde{s}\eta_{\epsilon}.
\end{equation}

The reference \cite{schmidt2007plate} shows that $\mbox{supp}(g'_{\epsilon}, g^{3}_{\epsilon}) \subset \Omega^{+}.$ Consider the ansatz
\begin{equation}\label{finalansatz}
\begin{split}
y^{h}=&(1-hK)y(x')+h[x_{3}(1-hC)\nu(x')+Qg_{\epsilon}(x')]+h^{2}QD^{h}(x',x_{3})\\
&+hQw_{\epsilon}^{h}(x')+h^{\frac{1}{2}}Q \begin{bmatrix}
0\\
0\\ 
v_{\epsilon}^{h}
\end{bmatrix}-h^{\frac{3}{2}} x_{3} Q \nabla (v_{\epsilon}^{h})-\frac{1}{2}h^{2}x_{3} Q\begin{bmatrix}
0\\
0\\
|\nabla v_{\epsilon}^{h}|^{2}
\end{bmatrix} +\\
& \quad h^{2}x_{3} Q (g_{\epsilon}'\nabla y^{T}\nabla \nu,0)  ,
\end{split}
\end{equation}
where $d^{h}(x',x_{3})$ is as in \eqref{defofdh}, $D^{h}(x',x_{3})=\int_{0}^{x_{3}} d^{h}(x',s) \, ds$, and $Q=(\nabla y | \nu)$ (which is no longer constant).  

We can compute:
\begin{equation}
\begin{split}
\nabla^{h}y^{h}&=Q+hQ\left[-CId+x_{3} \nabla y^{T} \nabla \nu+ \nabla' g_{\epsilon} + \begin{bmatrix}
g_{\epsilon}^{3} \nabla y^{T}\nabla \nu\\
-g_{\epsilon}' \nabla y^{T}\nabla \nu
\end{bmatrix}+d^{h}\otimes e_{3} - hCx_{3}\nabla y ^{T}\nabla \nu\right]\\
& \quad +h^{2}\nabla'(QD^{h})+hQ\nabla w_{\epsilon}^{h}+h^{\frac{1}{2}}Q\begin{bmatrix}
0_{2\times 2}& -\nabla (v_{\epsilon}^{h})^{T}\\
\nabla v_{\epsilon}^{h} & 0
\end{bmatrix}
-h^{\frac{3}{2}} x_{3}Q\nabla^{2}v_{\epsilon}^{h}\\
&\quad -\frac{1}{2}hQ\begin{bmatrix}
0\\
0\\
|\nabla v_{\epsilon}^{h}|^{2}
\end{bmatrix} \otimes e_{3}-\frac{1}{2}h^{2}x_{3}Q\begin{bmatrix}
0&0&0\\
0&0&0\\
\partial_{x_{1}}|\nabla v_{\epsilon}^{h}|^{2}&\partial_{x_{2}}|\nabla v_{\epsilon}^{h}|^{2}&0
\end{bmatrix}\\
&\quad +hQ (g_{\epsilon}'\nabla y^{T}\nabla \nu,0)\otimes e_{3}+x_{3}h^{2}\nabla'(Q (g_{\epsilon}'\nabla y^{T}\nabla \nu,0)). 
\end{split}
\end{equation}
In the above computation we have used that $v^{h}_{\epsilon}$ and $w^{h}_{\epsilon}$ are nonzero only in the region where $Q(x')$ is constant. Assume that
\begin{equation}\label{lastapproximationinansatz}
h \nabla' [Q(g_{\epsilon}' \nabla y^{T}\nabla \nu,0)] \to 0, 
\end{equation}
(we will prove this at the end using a retardation argument as before).

Before concluding, we need a few more technical remarks: note that
\begin{equation}
\begin{split}
\left| h^{\frac{1}{2}} \nabla^{2} v_{\epsilon}^{h} \right| &= h^{\frac{1}{2}} [\mathcal{O}(\nabla^{2} (\phi_{\epsilon} )v^{h})+\mathcal{O}((\nabla \phi_{\epsilon})( \nabla v^{h}))+\mathcal{O}(\phi_{\epsilon}( \nabla^{2} v^{h}))]\\
&= \textit{o} \left( \frac{h^{\frac{1}{2}}}{\epsilon^{2}} \right)+\mathcal{O} \left( \frac{h^{\frac{1}{2}}}{\epsilon} \right)+\textit{o} \left( 1 \right),
\end{split}
\end{equation}
while
\begin{equation}
\begin{split}
\left| \nabla v_{\epsilon}^{h} \right|&= \mathcal{O}(\phi_{\epsilon} \nabla v^{h})+\mathcal{O}(\nabla \phi_{\epsilon} v^{h})\\
&= \mathcal{O}(|\nabla v^{h}|)+\frac{1}{\epsilon} \mathcal{O}(v^{h}),
\end{split}
\end{equation}
and
\begin{equation}
\begin{split}
\left| \nabla w_{\epsilon}^{h} \right|&= \mathcal{O}(\phi_{\epsilon} \nabla w^{h})+\mathcal{O}(\nabla \phi_{\epsilon} w^{h})\\
&= \mathcal{O}(|\nabla w^{h}|)+\frac{1}{\epsilon} \mathcal{O}(w^{h}).
\end{split}
\end{equation}
Hence, by choosing $\epsilon$ appropriately depending on $h_{n}$, for example, $ \epsilon= \max (h^{\frac{1}{8}}, \sqrt{|v^{h}|}, \sqrt{|w^{h}|}) = \epsilon $, we can arrange that $\epsilon \to 0,$ as $h_{n} \to 0$ and 
\begin{equation}
\limsup \left| h^{\frac{1}{2}} \nabla^{2} v_{\epsilon}^{h} \right| + \left|   \nabla v_{\epsilon}^{h} \right| +  \left| \nabla w_{\epsilon}^{h} \right| < \infty.
\end{equation}

We now derive a uniform bound on several error terms. Using that $g=0$ on $\Omega^{0}$ we have that, on $\Omega^{0},$
\begin{equation}
\begin{split}
\nabla^{h}y^{h}&=Q+hQ\left[-C \text{Id} +d^{h}\otimes e_{3}\right]+h^{2}\nabla'(QD^{h})\\
&\quad +hQ\nabla w_{\epsilon}^{h}+h^{\frac{1}{2}}Q\begin{bmatrix}
0_{2\times 2}& -\nabla (v_{\epsilon}^{h})^{T}\\
\nabla v_{\epsilon}^{h} & 0
\end{bmatrix}
-h^{\frac{3}{2}} x_{3}Q\nabla^{2}v_{\epsilon}^{h}\\
&\quad -\frac{1}{2}hQ\begin{bmatrix}
0\\
0\\
|\nabla v_{\epsilon}^{h}|^{2}
\end{bmatrix} \otimes e_{3}-\frac{1}{2}h^{2}x_{3}Q\begin{bmatrix}
0&0\\
0&0\\
\partial_{x_{1}}|\nabla v_{\epsilon}^{h}|^{2}&\partial_{x_{2}}|\nabla v_{\epsilon}^{h}|^{2}
\end{bmatrix}\\
&= Q +h^{\frac{1}{2}}Q\begin{bmatrix}
0_{2\times 2}& -\nabla (v_{\epsilon}^{h})^{T}\\
\nabla v_{\epsilon}^{h} & 0
\end{bmatrix}+h\mathcal{O}(h^{\frac{1}{2}} \nabla^{2} v_{\epsilon}^{h}+|\nabla v_{\epsilon}^{h}|^{2}+\nabla w_{\epsilon}^{h}+1),
\end{split}
\end{equation}
and on $\Omega^{0} \cap F_{\epsilon}$, 
\begin{equation}
\begin{split}
\frac{1}{h} \left( (\nabla^{h} y^{h})^{T}\nabla^{h} y^{h} - Id \right) &= \mathcal{O} (h^{\frac{1}{2}} \nabla^{2} v_{\epsilon}^{h}+|\nabla v_{\epsilon}^{h}|^{2}+\nabla w_{\epsilon}^{h}+1)\\
&\leq C_{1}
\end{split}
\end{equation}
for some $C_{1} \in \mathbf{R}.$

We also have on $\Omega^{+}$, that
 \begin{equation}
\begin{split}
\nabla^{h}y^{h}&=Q+hQ\left[-C \text{Id}+x_{3} \nabla y^{T}\nabla \nu+ \nabla' g_{\epsilon} + \begin{bmatrix}
g_{\epsilon}^{3} \nabla y^{T}\nabla \nu\\
-g_{\epsilon}' \nabla y^{T}\nabla \nu
\end{bmatrix}+d^{h}\otimes e_{3}\right]\\
&\quad +h^{2}\nabla'(QD^{h}) +hQ (g_{\epsilon}'\nabla y^{T}\nabla \nu,0)\otimes e_{3}+x_{3}h^{2}\nabla'(Q (g_{\epsilon}'\nabla y^{T}\nabla \nu,0))\\
& \quad -h^{2}x_{3}C\nabla'\nu ,
\end{split}
\end{equation}
 
	and on $\Omega_{+} \cap F_{\epsilon}$,
\begin{equation}
\begin{split}
\frac{1}{h} \mbox{sym}\left(   Q^{T}\nabla^{h} y^{h}-Id  \right) &= -C \text{Id}+x_{3}\mathbf{II}+ \widetilde{s}\eta_{\epsilon}+\mbox{sym}(d^{h}\otimes e_{3})\\
&\quad +h\nabla'[QD^{h}]+x_{3}h\nabla'[Q (g_{\epsilon}'\nabla y^{T}\nabla \nu,0)]+o(1) \\
&\leq C_{2},
\end{split}
\end{equation}
for some $C_{2} \in \mathbf{R}$ since all terms are bounded. 

Finally, using the approximation arguments presented earlier in this section along with the fact that the measure of $F_{\epsilon}$ tends to $0$, we have that 
\begin{equation}
 \lim_{h \to 0} \frac{1}{h^{2}} \int_{\Omega^{0}}W(x,\nabla^{h} y^{h}\left( A^{h} \right)^{-1}(x))=\int_{\Omega^{0}}Q_{2}(x,[x_{3}\nabla y^{T} \nabla \nu +s(x')-B_{2\times 2}(x))] ).
\end{equation}

Similarly, and using \eqref{lastapproximationinansatz} along with the previous arguments we have that 

\begin{equation}
\begin{split}
 \lim_{h \to 0}\frac{1}{h^{2}}\int_{\Omega_{+}}W(x,\nabla^{h} y^{h}\left( A^{h} \right)^{-1}(x))&= \lim_{h \to 0}\frac{1}{h^{2}} \int_{\Omega_{+}}W(x,Q^{T}\nabla^{h} y^{h}\left( A^{h} \right)^{-1}(x))\\
&=\int_{\Omega_{+}}Q_{2}(x,[x_{3}\nabla y^{T} \nabla \nu +s(x')-B_{2\times 2}(x)] ).
\end{split}
\end{equation}

It remains to prove that we may take a sequence such that 
\begin{equation}
h\nabla'[Q (g_{\epsilon}'\nabla y^{T}\nabla \nu,0)]  \to 0, 
\end{equation}
we proceed with a retardation argument: let $h_{n} \to 0$ monotonically, we define a sequence $\sigma(n)$ as 
\begin{equation*}
\sigma(1)=1
\end{equation*}
and
\begin{equation*}
\sigma(n+1)=
\begin{cases}
\sigma(n)+1 \mbox{ if } h_{n} \nabla'[Q (g_{\epsilon}'\nabla y^{T}\nabla \nu,0)]  \leq \frac{1}{\sigma(n)+1}  \\
\sigma(n) \mbox{ if not}.
\end{cases}
\end{equation*}
Then our final ansatz is the original given by \eqref{finalansatz} with $h_{n}$ as in the original sequence, but $\epsilon_{n}$ replaced by $\epsilon_{\sigma(n)}.$ This works provided $h_{n} \to 0$ monotonically, which can be assumed without loss of generality by taking a sub-sequence.
%\begin{equation}
%\begin{split}
%y^{h_{n}}=&(1-h_{n}K)y(x')+h_{n}[x_{3}(1-h_{n}K)\nu(x')+Qg_{\epsilon_{\sigma(n)}}(x')]+h_{n}^{2}QD^{h_{(n)}}(x',x_{3})\\
%&+h_{n}Qw_{\epsilon_{\sigma(n)}}^{h_{(n)}}(x')+h_{n}^{\frac{1}{2}}Q \begin{bmatrix}
%0\\
%0\\ 
%v_{\epsilon_{\sigma(n)}}^{h_{(n)}}
%\end{bmatrix}-h^{\frac{3}{2}} x_{3} \nabla (Qv_{\epsilon_{\sigma(n)}}^{h_{(n)}})-h^{2}x_{3} Q\begin{bmatrix}
%0\\
%0\\
%|\nabla v_{\epsilon_{\sigma(n)}}^{h_{(n)}}|^{2}
%\end{bmatrix} -x_{3}h^{2}\nabla'Q (g_{\epsilon}'\nabla y^{T}\nabla \nu,0)   ,
%\end{split}
%\end{equation}

To conclude the upper bound, we proceed in the following way:
\begin{itemize}
\item[$\bullet$] Approximate the limiting deformation $y(x')$ by $y_{\delta} \in \mathcal{A}_{0}$ such that the difference in the energy is less than $\frac{\delta}{3}.$

\item[$\bullet$] Approximate the optimal $s$ by $s_{\delta} \in C^{\infty}_{0}$ such that the difference in the energy is less than $\frac{\delta}{3}.$

\item[$\bullet$] Choose $h_{n}$ such that $y^{h_{n}}$ achieves the energy associated to $y_{\delta}$ and $s_{\delta}$ up to an error of $\frac{\delta}{3}.$
\end{itemize}
This way we construct a sequence that converges to $y(x')$ in the right way  and achieves the lower bound.

\end{proof}

\section{Algebraic reduction}

Overall, the goal of this section is to understand and simplify the functional \eqref{Gammalimitfunctional}. We start by finding an explicit formula for $s(x'),$ given by
%\begin{equation}
%s(x')=\mbox{argmin}_{s \in \mathbf{R}^{2\times 2}, d \in \mathbf{R}^{3}} \int_{-\frac{1}{2}}^{\frac{1}{2}} {Q}_{3}\Big( x,\overline{A}^{-1}(x')[s+x_{3}\nabla\emph{y}^{T}\nabla\emph{b}(x')-(\overline{A}B)_{2\times 2}(x',x_{3})] + d \otimes e_{3}\overline{A}^{-1}(x') \Big) dx_{3},
%\end{equation}
\begin{equation}
s(x') = \mbox{argmin}_{s \in \mathbf{R}^{2\times 2}} \int_{-\frac{1}{2}}^{\frac{1}{2}} Q_{2}(x',t,s+x_{3}\nabla\emph{y}^{T}\nabla\emph{b}(x')-(\overline{A}B)_{2\times 2}(x',x_{3}),\overline{A}) dt
\end{equation}
where $Q_{2}$ is given by \eqref{definition of Q_2}. We proceed to reduce the lower bound of the problem to an elastic sheet with thickness-independent elastic law, and simpler prestrain. In the cases treated by Proposition 2 and Theorem 3, the analysis in this section is a reduction of the $\Gamma$ limit, while in general it is only a reduction of the lower bound. The new form of the energy will involve the projection of the prestrain onto a suitable linear space with the appropriate norm.  

Our basic approach is the following: first we eliminate the dependence of the lower bound on $s(x')$ by writing it as the energy of a new quadratic form and prestrain. Next, we rewrite this energy as that of a constant-in-thickness quadratic form and linear-in-thickness prestrain, plus an configuration-independent term. Lastly we analyze the configuration-independent term. Similar considerations are found in \cite{de2020energy}.

%In previous sections, equations \eqref{definition of Q_2} defined the quadratic form $Q_{2}$ when the prestrain equals the identity at leading order. We now extend that definition to the case of an arbitrary prestrain: for $X \in s_{2 \times 2}$ we define 
%\begin{equation}\label{defQ2}
%Q_{2}(x',t,X)=\min_{c \in \mathbf{R}^{3}} Q_{3}(x',t,\overline{A}^{-1}(x')[X+c \otimes e_{3}]\overline{A}^{-1}(x')),
%\end{equation}
%Note that the function
%\begin{equation}
%X \mapsto c_{min}(X)
%\end{equation}
%is linear \cite{bhattacharya2016plates}.

In this section, in order to ease notation, we will omit the dependence on $\overline{A}$. Note that since the form $Q_{2}(x',t, \cdot)$ is bilinear, there is a tensor $L_{2}(x',t)$ such that
\begin{equation}
Q_{2}(x',t,X)=\langle L_{2}(x',t)X,X \rangle.
\end{equation}
An immediate computation shows that
\begin{equation}
s(x')=(L_{2}^{*})^{-1} \left( \int_{-\frac{1}{2}}^{\frac{1}{2}} L_{2}(x',t)([-t\nabla y^{T} \nabla b + (\overline{A}B)_{2\times 2}]) \right),
\end{equation}
where $L_{2}^{*}(x')=\int_{-\frac{1}{2}}^{\frac{1}{2}} L_{2}(x',t)dt$. We can write $s(x')$ in terms of the tensors $\phi_{1}(x'):\mathbf{R}^{2 \times 2} \to \mathbf{R}^{2 \times 2}$ and $\phi(x'):L^{2}[(-\frac{1}{2}, \frac{1}{2}), \mathbf{R}^{2 \times 2}] \to \mathbf{R}^{2 \times 2}$ defined as 
\begin{equation}
\phi_{1}(x')(X)=(L^{*}_{2})^{-1} \left( \int_{-\frac{1}{2}}^{\frac{1}{2}} tL_{2}(x',t)X dt \right)
\end{equation}
and 
\begin{equation}
\phi(x')(X)=(L^{*}_{2})^{-1} \left( \int_{-\frac{1}{2}}^{\frac{1}{2}} L_{2}(x',t)X(t) dt \right).
\end{equation}
The tensor $\phi(x')(X(t))$ gives the projection of $X(t)$ onto the space of functions constant in $t,$ and $\phi_{1}(X)=\phi(tX)$. By writing $M$ in terms of these tensors, we can rewrite the lower bound of the $\Gamma-$limit (equation \ref{Gammalimitfunctional}) as 
\begin{equation}
\begin{split}
{\mathcal{I}}(y)=\frac{1}{2}\int_{\Omega^{1}} \langle L_{2}(x',t)[ t\nabla y^{T}\nabla b-\phi_{1}(\nabla y^{T}\nabla b)-\left([\overline A(x')B(x)]_{2\times 2}-\phi(\overline AB)_{2\times 2}] \right), \\
t\nabla y^{T}\nabla b-\phi_{1}(\nabla y^{T}\nabla b)- \left( [\overline A(x')B(x)]_{2\times 2}-\phi(\overline AB)_{2\times 2} \right) \rangle \,  dx'dt.
\end{split}
\end{equation}
Let 
\begin{equation}
\begin{split}
V_{2}(x',t)&=tId-\phi_{1}(x') \\
N_{1}(x',t)&= [\overline A(x')B(x',t)]_{2 \times 2}-\phi([\overline A(x')B(x',t)]_{2 \times 2})\\
T_{2}(x',t)&= V_{2} \circ L_{2} \circ V_{2},
\end{split}
\end{equation}
then
\begin{equation}\label{simplificationofI}
\begin{split}
{\mathcal{I}}(y)&=\frac{1}{2}\int_{\Omega^{1}} \langle L_{2}(x',t)[V_{2} \nabla y^{T}\nabla b-N_{1}(x)],V_{2} \nabla y^{T}\nabla b-N_{1}(x) \rangle \, dx'dt\\
=&\frac{1}{2}\int_{\Omega^{1}} \langle T_{2}(x',t)[ \nabla y^{T}\nabla b-n^{*}(x')],
 \nabla y^{T}\nabla b-n^{*}(x') \rangle \, dx'dt\\
+&\frac{1}{2}\int_{\Omega^{1}}\langle L_{2}(x',t)N_{1}, N_{1} \rangle - \langle T_{2}(x',t)n^{*}(x'), n^{*}(x') \rangle \, dx'dt,
\end{split}
\end{equation}
where $n^{*}(x')$ satisfies
\begin{equation}\label{defn*}
\int_{-\frac{1}{2}}^{+\frac{1}{2}}  V_{2}(x',t)\circ L_{2}(x',t)N_{1}dt =\int_{-\frac{1}{2}}^{+\frac{1}{2}}  T_{2}(x',t)[n^{*}(x')]dt
\end{equation}
i.e.
\begin{equation}
n^*(x')=(T_{2}^{*})^{-1} \left( \int_{-\frac{1}{2}}^{\frac{1}{2}}  V_{2}(x',t)\circ L_{2}(x',t)N_{1}dt \right)
\end{equation}
%\begin{equation}
%\int_{-\frac{1}{2}}^{\frac{1}{2}} t T(x',t)N(x',t)dt=T^{*}_{2}\left( \int_{-\frac{1}{2}}^{\frac{1}{2}} t N^{*}(x',t) dt \right),
%\end{equation}
where $T_{2}^{*}=\int_{-\frac{1}{2}}^{\frac{1}{2}} T_{2}(x',t)dt.$ We can interpret $n^{*}(x')$ as the preferred curvature of the sheet. Note that $T_{2}^{*}$ is positive definite, and therefore invertible, since for any $X \neq 0$ we have that
\begin{equation}
\begin{split}
\langle T_{2}^{*} X, X \rangle &= \int_{-\frac{1}{2}}^{\frac{1}{2}} \langle T_{2}(x',t)X, X \rangle \, dt\\
&= \int_{-\frac{1}{2}}^{\frac{1}{2}} \langle L_{2}(x',t)\left( tX-\phi_{1}(x')X \right), \left( tX-\phi_{1}(x')X \right) \rangle \, dt \\
& > 0.
\end{split}
\end{equation}
The last expression is positive since $L_{2}$ is positive definite and $tX-\phi_{1}(x')X$ being $0$ implies that $X$ is 0.

Now we can rewrite the lower bound of the limit as the integral of a thickness-independent quadratic form, indeed, we can rewrite the configuration-dependent part of \eqref{simplificationofI} in terms of  $T_{2}^{*}(x')$:
\begin{equation}
\begin{split}
&\frac{1}{2}\int_{-\frac{1}{2}}^{\frac{1}{2}} \langle T_{2}(x',t)[ \nabla y^{T} \nabla b - n^{*}(x')], \nabla y^{T} \nabla b - n^{*}(x') \rangle dt\\
=& \frac{1}{2} \langle {T}_{2}^{*}(x')[ \nabla y^{T} \nabla b - n^{*}(x')], \nabla y^{T} \nabla b - n ^{*}(x') \rangle.
\end{split}
\end{equation}
%It's easy to check that $T^{*}_{2}$ satisfies 
%\begin{equation}
%\overline{T}_{2}=\int_{-\frac{1}{2}}^{\frac{1}{2}} T_{2}(x',t)dt.
%\end{equation} 
Let $R(x')$ be the residue, i.e. the configuration-independent part of \eqref{simplificationofI}: 
\begin{equation}
R(x')=\int_{-\frac{1}{2}}^{\frac{1}{2}}\langle  L_{2}(x',t)N_{1}, N_{1}  \rangle dt-\int_{-\frac{1}{2}}^{\frac{1}{2}}\langle T_{2}(x',t)n^{*}(x'), n^{*}(x') \rangle dt.
\end{equation}
%the terms in equation \eqref{1reduction} as follows: formally, $tn^*(x')$ is the projection of $N(x)$ onto the space of funtions linear in $t,$ with the inner product $\langle f,g \rangle=\int \langle T(x)f(x),g(x) \rangle.$ And $R(x')$ is the norm of $N(x)$ minus the projection. 

The configuration-independent term is actually the norm squared of the prestrain minus its projection onto the space of functions affine in $t.$ We can see this by arguing as follows: note that $\phi(Y)$ gives the projection of $Y$ onto the space of functions constant in $t,$ with the inner product $\langle f,g \rangle=\int_{\Omega^{1}} \langle L_{2}(x)f,g\rangle dx$. Note also that, for any $2 \times 2$ matrix $X,$ we have
\begin{equation}
V_{2}(x',t)(X)=tX-\phi(x')(tX).
\end{equation}
From equation \eqref{defn*} we have that, for any $2 \times 2$ matrix $X,$ 
\begin{equation}
\int_{-\frac{1}{2}}^{\frac{1}{2}} \langle X, V_{2}(x',t)\circ L_{2}(x',t)N_{1} \rangle dt=\int_{-\frac{1}{2}}^{\frac{1}{2}} \langle X, T_{2}(x',t)[n^{*}(x')] \rangle dt,
\end{equation}
and hence
\begin{multline}
\int_{-\frac{1}{2}}^{\frac{1}{2}} \langle tX-\phi(tX),L_{2}((\overline{A}B)_{2 \times 2}-\phi((\overline{A}B)_{2\times 2})) \rangle dt=\\
\int_{-\frac{1}{2}}^{\frac{1}{2}} \langle tX-\phi(tX),L_{2}(tn^*(x')-\phi(tn^*(x')) \rangle dt.
\end{multline}
Since
\begin{equation}
\int_{-\frac{1}{2}}^{\frac{1}{2}} \langle tX-\phi(tX),L_{2}(\phi((\overline{A}B)_{2\times 2})) \rangle dt=0,
\end{equation}
we have that 
\begin{equation}
\int_{-\frac{1}{2}}^{\frac{1}{2}} \langle tX-\phi(tX),L_{2}((\overline{A}B)_{2\times 2}) \rangle dt=\int_{-\frac{1}{2}}^{\frac{1}{2}} \langle tX-\phi(tX),L_{2}(tn^*(x')-\phi(tn^*(x')) \rangle dt.
\end{equation}
Note that matrix fields of the form $tX-\phi(tX)$ span, as $X$ varies in the space of $2\times 2$ matrices, the vector space orthogonal to constants in the space of matrix fields affine in $t.$ Hence, $tn^*(x')-\phi(tn^*(x'))$ is the projection of $(\overline{A}(x')B(x))_{2\times 2}$ onto the orthogonal complement of functions constant in $t$ in the space of functions affine in $t.$ We can rewrite the residue as 
\begin{equation}
\begin{split}
R&=\int_{-\frac{1}{2}}^{\frac{1}{2}} \langle L_{2}((\overline A B)_{2 \times 2}-\phi((\overline A B)_{2 \times 2})),(\overline A B)_{2 \times 2}-\phi((\overline A B)_{2 \times 2}))\rangle dt\\
&-\int_{-\frac{1}{2}}^{\frac{1}{2}} \langle L_{2}(tn^*(x')-\phi(tn^*(x')),tn^*(x')-\phi(tn^*(x'))\rangle dt\\
&= \int_{-\frac{1}{2}}^{\frac{1}{2}} \langle L_{2}((\overline A B )_{2 \times 2},(\overline A B)_{2 \times 2}\rangle dt-\int_{-\frac{1}{2}}^{\frac{1}{2}} \langle L_{2}(\phi((\overline A B) _{2 \times 2})),\phi((\overline A B )_{2 \times 2})\rangle dt\\
&-\int_{-\frac{1}{2}}^{\frac{1}{2}} \langle L_{2}(tn^*(x')-\phi(tn^*(x')),tn^*(x')-\phi(tn^*(x'))\rangle dt,
\end{split}
\end{equation}
In the last equality, we have used the definition of $\phi,$ and the fact that $L_{2}$ is symmetric. Note that the first item being substracted is the norm squared of the projection of $(\overline{A}(x')B(x))_{2\times 2}$ onto the space of functions constant in $t,$ while the second item being substracted is the norm squared of the projection of $(\overline{A}(x')B(x))_{2\times 2}$ onto the orthogonal complement of this subspace in the space of functions affine in $t.$ Therefore, $R$ is the norm squared of $(\overline{A}(x')B(x))_{2\times 2}$ minus its projection onto the space of functions affine in $t.$

%Another interpretation is as follows: after unraveling the algebra, we can see that $n^*$ satisfies
%\begin{equation}
%\int_{-\frac{1}{2}}^{+\frac{1}{2}}L_{2}(tn^*-\phi(tn^*))dt=\int_{-\frac{1}{2}}^{+\frac{1}{2}}L_{2}(A(x')B(x)-\phi(A(x')B(x)))dt.
%\end{equation}
%Note that $\phi(X)$ gives the projection of $X$ onto the space of functions constant in $t$ with the inner product $\langle f(x) , g(x) \rangle=\int \langle L_{2}(x)f(x) , g(x)\rangle dx. $ Hence, $tn^*$ is the linear part of the projection of $A(x')B(x)$ onto the space of functions affine in $t,$ with the $L_{2}$ inner product (note that constants are not orthogonal to linear functions in this inner product). Hence, the reside is the norm of $A(x')B(x)$ minus the linear part of its projection onto the space of functions affine in $t.$

%and $N^{*}$ satisfies 
%
%hence we can choose $N^{*}(x',t)=tn^{*}(x'),$ where
%\begin{equation}
%n^*(x')=12(T_{2}^{*})^{-1} \left( \int_{-\frac{1}{2}}^{\frac{1}{2}} t T(x',t)N(x',t)dt \right)
%\end{equation}

\section{Necessity of an $h$ scale blowup}

In order to match the metric at order $1,$ an ansatz must include the terms $y(x')$ and $x_{3}\nu(x'),$ but there are terms other than this that are much bigger than thickness, leading to wrinking phenomena. Is it possible to construct an ansatz that does not have this feature? On the one hand, it is natural to ask whether it is possible to achieve the lower bound by an ansatz that contains only terms of order $h$ (execpt for the terms $y(x')$ and $x_{3}\nu(x')$). It is also natural from a physics perspective, since showing that such a loss of compactness is inevitable is evidence that such deformations can occur in experiments. In this section, we will prove that, in general, a minimizing sequence contains terms of order bigger than $h$ if $\mathbf{II}=0$. We again assume $A(x')=Id.$

Let $\eta^{h}: \Omega \to \mathbf{R}^{3}$ be such that
\begin{equation}
(\nabla \eta^{h})^{T}\nabla \eta^{h}=Id_{2\times 2}.
\end{equation}

\begin{theorem}
Let $y^{h}: \Omega^{1} \to \mathbf{R}^{3}$ be such that
\begin{equation}\label{h-compactness}
\frac{1}{h^{2}} \mathcal{E}^{h}(y^{h}) \leq k,
\end{equation}
assume that
\begin{equation}\label{hyp-extra}
\parallel \nabla^{h}y^{h} - \nabla^{h}\phi^{h}\parallel^{2} \leq k h^{2},
\end{equation}
for some $\phi,$ where $\phi^{h}: \Omega^{1} \to \mathbf{R}^{3}$ is defined as
\begin{equation}
\phi^{h}(x_{1}, x_{2}, x_{3})= \eta^{h}(x')+hx_{3} \nu^{h}(x').
\end{equation}
Assume that $y^{h} \to y$ (independent of $z$) where $y|_{\Omega \times \{ 0\}}$ is an isometric immersion. Let $\omega \subset \Omega$ be a region such that $\mathbf{II}|_{\omega}=0,$ where $\mathbf{II}$ is the second fundamental form of $y,$
then
\begin{equation}
\begin{split}
&\liminf  \frac{1}{h^{2}} \int_{\omega^{1}} W(x,\nabla^{h} y^{h} A^{h}(x)) \geq \\
&\frac{1}{2}\inf_{g \in W^{1,2}(\Omega,\mathbf{R}^{3})} \int_{\Omega^{1}} Q_{2} \left( x, -B_{2\times 2}(x)+\nabla g(x') \right) dx'dx_{3}.
\end{split}
\end{equation}
\end{theorem}

\begin{remark}
%In particular, if $\overline{A}=Id$ and $\mathbf{II}=0,$ then
%\begin{equation}\label{uppbound}
%\begin{split}
%&\liminf \frac{1}{h^{2}} \mathcal{E}^{h}(y^{h}) \geq \\
%&\inf_{g \in W^{1,2}(\Omega,\mathbf{R}^{3})} \int_{\Omega^{1}} Q_{2} \left( x, x_{3}\nabla\textbf{y}^{T}\nabla \textbf{b}-\overline{A}(x')B_{2\times 2}(x)+\nabla g(x') \right) dx'dx_{3}.
%\end{split}
%\end{equation}
%Note that 
This is in general a strictly higher bound than \eqref{definition of Q_2'}, since it amounts to restricting $s(x')$ to gradient fields.
\end{remark}

Let $\overline{y}^{h}$ and $\overline{\phi}^{h}$ be the rescaled versions of $y^{h}$ and $\phi^{h}$ to the domain $\Omega^{h}.$

Note that \eqref{hyp-extra} implies
\begin{equation} \label{hsquarebound2}
\parallel \nabla(\overline{y}^{h}-\overline{\phi}^{h}) \parallel_{L^{2}(\Omega^{h})}^{2 } \leq Kh^{3},
\end{equation}
Let $\overline{f}^{h}: \Omega^{h} \to \mathbf{R}^{3}$ be such that 
\begin{equation}
\overline{y}^{h}=\overline{\phi}^{h}+h\overline{f}^{h},
\end{equation}
and let $f^{h}: \Omega^{1} \to \mathbf{R}^{3}$ be the unrescaled version of $\overline{f}^{h}.$ Note that \eqref{hsquarebound2} implies that 
\begin{equation}
\parallel \nabla^{h} f^{h} \parallel_{L^{2}} \leq C ,
\end{equation}
we immediately get 
\begin{equation}
\parallel \nabla f^{h} \parallel_{L^{2}} \leq \parallel \nabla^{h} f^{h} \parallel_{L^{2}} \leq C,
\end{equation}
hence by Relich-Kondrachov there is $f \in W^{1,2}(\Omega^{1}, \mathbf{R}^{3})$ such that 
\begin{equation}
f^{h} \rightharpoonup f
\end{equation}
weakly in $W^{1,2}.$ We also get that 
\begin{equation}
\nabla^{h} f^{h} \rightharpoonup \left( \partial_{1} f, \partial_{2} f, b \right),
\end{equation}
for some $b \in L^{2}(\Omega^{1}, \mathbf{R}^{3}).$

By \cite{bhattacharya2016plates} we now that there exists a rotation valued field $R(x') \in W^{1,2}(\Omega, M^{3\times 3})$ such that 
\begin{equation}
\parallel \nabla^{h}y^{h}-R(x')\parallel_{L^{2}} \leq C h^{2}.
\end{equation}
We now argue as in the lower bound, and define
\begin{equation}
S_{h}=\frac{1}{h} \left( R(x')\nabla^{h}y^{h}- Id\right),
\end{equation}
then we have that $S_{h} \rightharpoonup S(x',x_{3})$ weakly in $L^{2},$ where $S$ satisfies 
\begin{equation}
\left( S(x',x_{3}) \right)_{2\times 2} = s(x'),
\end{equation}
since $\mathbf{II}=0.$

Arguing as in section 4, we know that
\begin{equation}\label{twoside}
\begin{split}
\frac{1}{h} \left(\nabla^{h}y^{h}\nabla^{h}y^{h}-Id \right)&=2S_{h}+hS_{h}S_{h}^{T}\\
&\rightharpoonup 2 S
\end{split}
\end{equation}
weakly in $L^{1},$ since by Holder $\parallel S_{h}S_{h}^{T} \parallel_{L^{1}} \leq\parallel S_{h} \parallel_{L^{2}}^{2}\leq k.$ 

On the other hand,
\begin{equation}\label{oneside}
\begin{split}
\frac{1}{h} \left((\nabla^{h}y^{h})^{T}\nabla^{h}y^{h}-Id \right)&=2\mbox{sym} \left( (\nabla \eta^{h})^{T}\nabla' \nu^{h}+ (\nabla^{h} \eta^{h})^{T} \nabla^{h}f^{h}(x',x_{3}) \right)\\
&+h\left( (\nabla^{h}f^{h})^{T}\nabla^{h}f^{h}+(\nabla^{h}f^{h})^{T}\nabla'\nu^{h}+(\nabla'\nu^{h})^{T}\nabla^{h}f^{h} \right) \\
& \quad +h \left( (\nabla'\nu^{h})^{T}\nabla'\nu^{h} \right).
\end{split}
\end{equation}
We know that $\nabla^{h} \phi^{h} \to R$ strongly in $L^{2},$ where $R: \Omega \to SO(3)$ is a rotation-valued field. We also know that $R$ is constant in $\omega,$ since $\mathbf{II}=0.$ After a change of coordinates, we may assume $R=Id,$ we also have that $\mbox{sym} \left( \nabla^{y} (\eta^{h})^{T}\nabla \nu^{h} \right) \rightharpoonup 0$ in $\omega$ (we will prove this in a moment). Using Holder's inequality once again, we get 
\begin{equation}
\begin{split}
&\parallel  (\nabla^{h}f^{h})^{T}\nabla^{h}f^{h}+(\nabla^{h}f^{h})^{T}\nabla'\nu^{h}+(\nabla'\nu^{h})^{T}\nabla^{h}f^{h}+(\nabla'\nu^{h})^{T}\nabla'\nu^{h} \parallel_{L^{1}}\\
\leq &C \left( \parallel  (\nabla^{h}f^{h})\parallel_{L^{2}}^{2}+\parallel  (\nabla'\nu^{h})\parallel_{L^{2}}^{2} \right)\\
\leq &k 
\end{split}
\end{equation}

Hence
\begin{equation}
\frac{1}{h} \left((\nabla^{h}y^{h})^{T}\nabla^{h}y^{h}-Id \right) \rightharpoonup \mbox{sym} \left( \partial_{1}f, \partial_{2} f, b \right),
\end{equation}

equating \eqref{oneside} and \eqref{twoside} we get
\begin{equation}
s(x')=(\nabla f)_{2\times 2}.
\end{equation}
In particular, $(\nabla f)_{2\times 2}$ is independent of $x_{3}.$
% and $\nabla f^{h} \rightharpoonup \nabla f$ weakly in $W^{1,2}$ for some $f(x',x_{3}).$  
%
%Combining the previous results we have that the upper $2\times 2$ block of $\nabla f$ depends only on $x'$:
%\begin{equation}
%\left( \nabla f(x',x_{3})\right)_{2\times 2}=\left( \nabla f(x')\right)_{2\times 2}.
%\end{equation}

Using results from previous sections we have that
\begin{equation}
\frac{1}{h^{2}}\liminf \mathcal{E}(y^{h}) \geq \int_{\Omega } Q_{2} \left( x, -B_{2\times 2}(x)+\nabla f(x') \right) dx'dx_{3}. 
\end{equation}

In order to show that $\mbox{sym} (\nabla^{h} (\eta^{h})^{T}\nabla \nu^{h}) \to 0$ in $\omega$, we note that because of \eqref{h-compactness}, we have that
\begin{equation}
\frac{1}{h^{2}}\mathcal{E}(\phi^{h})\leq k < \infty
\end{equation}
and therefore there exist rotations $\widetilde{R}^{h}(x')$ such that
\begin{equation}
\parallel \nabla^{h}\phi^{h}-\widetilde{R}^{h}(x') \parallel_{L_{2}} \leq C h^{2},
\end{equation}
and that if we define
\begin{equation}
\overline{S}^{h}_{\phi}=\frac{1}{h} \left( \widetilde{R}^{h}\nabla^{h}\phi^{h}-Id_{3\times 3} \right), 
\end{equation}
then $\overline{S}_{\phi}^{h} \rightharpoonup \overline{S}_{\phi}$ weakly in $L^{2},$ for some $\overline{S}_{\phi}$. We claim that $\nabla \nu^{h}$ is uniformly bounded in $L_{2}.$ Since $\overline{S}_{\phi}^{h}$ is uniformly bounded in $L^{2},$ we have
\begin{equation}
\overline{S}_{\phi}^{h}-\int_{-\frac{1}{2}}^{\frac{1}{2}}\overline{S}_{\phi}^{h} dx_{3}
\end{equation}
is uniformly bounded in $L^{2},$ but
\begin{equation}
\overline{S}_{\phi}^{h}-\int_{-\frac{1}{2}}^{\frac{1}{2}}\overline{S}_{\phi}^{h} dx_{3}=x_{3}\widetilde{R}^{h} \nabla \nu^{h},
\end{equation}
since multiplication by $\widetilde{R}^{h}$ does not change the $L_{2}$ norm, we have that $\nabla \nu^{h}$ is bounded in $L^{2}$. Furthermore, $\overline{S}_{\phi}$ satisfies that
\begin{equation}
(\overline{S}_{\phi})_{2\times 2}=s_{\phi}(x')+x_{3}\nabla y ^{T} \nabla \nu,
\end{equation}
where $\nu$ is the Cosserat vector (unit normal) of $y$. From this, we get
\begin{equation}
\begin{split}
\left( \frac{(\nabla^{h}\phi^{h})^{T}\nabla^{h}\phi^{h}-Id}{h} \right)_{2\times 2}&=\left( ((\overline{S_{\phi}}^{h})^{T}+\overline{S}_{\phi}^{h})+ h(\overline{S_{\phi}}^{h})^{T}(\overline{S_{\phi}}^{h})\right)_{2\times 2}\\
&\rightharpoonup 2\left( s_{\phi}(x')+x_{3}\mbox{sym}(\nabla y ^{T} \nabla \nu) \right)
\end{split}
\end{equation}
weakly in $L^{1}$ .On the other hand, we know that
\begin{equation}
\left( \frac{(\nabla^{h}\phi^{h})^{T}\nabla^{h}\phi^{h}-Id}{h} \right)_{2\times 2}=2\mbox{sym}\left( (\nabla' \eta)^{T}\nabla \nu^{h} \right)+h(\nabla \nu^{h})^{T}\nabla \nu^{h}.
\end{equation}

Since $\parallel \nabla \nu^{h}\parallel_{L^{2}}$ is uniformly bounded, we have that $\nabla \nu^{h} \rightharpoonup \nabla \nu'$ in $L^{2}$ for some $\nu' \in W^{1,2}(\Omega).$ \footnote{Of course, we are using that the weak limit of a gradient is a gradient}. By Holder, we have that $(\nabla \nu^{h})^{T}\nabla \nu^{h}$ is uniformly bounded in $L^{1},$ and therefore   
\begin{equation}
\begin{split}
&2\mbox{sym}\left( (\nabla' \eta)^{T}\nabla {\nu}^{h} \right)+h(\nabla {\nu}^{h})^{T}\nabla {\nu}^{h}\\
\rightharpoonup & \nabla' \eta^{T}\nabla {\nu}'.
\end{split}
\end{equation}
Therefore
\begin{equation}
\mbox{sym}((\nabla' \eta)^{T}\nabla {\nu}^{h}) \rightharpoonup \mbox{sym}\nabla y ^{T} \nabla \nu=0.
\end{equation}
hence
\begin{equation}
\begin{split}
s_{\phi}&=0,\\
\mbox{sym}\nabla \nu'&=\mbox{sym}\nabla \nu,
\end{split}
\end{equation}
since $(\nabla' y|\nu)$ is the identity in $\omega.$
 
 \section{Prestrain with variations in thickness of different order}

So far, we have assumed that the prestrain is of order $h$ (or in other words that we can write $A^{h}(x',x_{3})=\overline{A}(x')+hB^{h},$ where the matrix $B^{h}(x',x_{3})$ is bounded). This section is devoted to analyzing prestrains that are much bigger than the thickness. As before, we will reduce ourselves to cases in which the limiting deformation is an isometric immersion. First we ask the question: if a sequence of minimizers has finite bending energy then is the prestrain is of order $h$? in other words, does
\begin{equation}
\frac{1}{h^{2}} \int_{\Omega^{1}} W(x',x_{3}, \nabla^{h}y^{h} \left( A^{h} \right)^{-1}) \leq C < \infty,
\end{equation} 
imply that $A^{h}(x',x_{3})=\overline{A}(x')+hB^{h}(x',x_{3}),$ where $B^{h}(x',x_{3})$ is bounded (in some $L^{p}$ norm)? This implication is not true, as will be proved shortly. Our motivation for asking this question lies not only in its physical interest, but also in our opinion that the question of whether order $h$ prestrain is a necessary condition for finite bending energy is mathematically interesting in its own right. 

This section also contains three main examples that illustrate the possible pathological behavior if the prestrain is not of order $h:$ first we show that if the prestrain is much bigger than the thickness, a limiting configuration may not exist in the strong sense. The second example shows that even if a limiting configuration exists, it may not be $W^{2,2}.$ Lastly, we show that even if the limiting configuration exists and is $W^{2,2},$ the prestrain may be much bigger than the thickness, and the curvature of a minimizing sequence may blow up (the idea is to construct a sequence which converges to a limit in $W^{2,2},$ but convergence is not in the $W^{2,2}$ topology). Apart from exploring the possible pathological behavior of larger prestrains, these examples also show that the implication considered in the first paragraph is not true even if the hypotheses are significantly strengthened. It is even possible for all the conclusions of Theorem \ref{Gammaconvergencetheorem1} to hold, with all the hypotheses being valid except that the prestrain is of order $h$. 

Finally, we prove a weaker version of the implication considered in the first paragraph: if
\begin{equation}
\frac{1}{h^{2}} \int_{\Omega^{1}} W(x',x_{3}, \nabla^{h}y^{h} \left( A^{h} \right)^{-1}) \leq C < \infty,
\end{equation} 
and a limiting deformation exists in the strong $W^{1,2}$ sense, then $A^{h}(x',x_{3})=\overline{A}(x')+B^{h}(x',x_{3}),$ then $B^{h}(x',x_{3}) \to 0.$

We start with a simple example of a sequence with finite (in fact, zero) bending energy and infinite prestrain/thickness ratio. 

\begin{example}
Let $\Omega=[0,1]^{2}$
\begin{equation}
A^{h}(x)=Id+\frac{z}{h^{\alpha}} 
\begin{bmatrix}
    1 & 0 & 0  \\
    0 & 0 & 0  \\
    0 & 0 & 0 
  \end{bmatrix} 
\end{equation}    
   and $W(X)=\parallel \sqrt{X^{T}X}-Id \parallel^{2}.$ Then there exists a sequence $u^{h}:\Omega^{h} \to \mathbf{R}^{3}$ such that
\begin{equation}
(A^{h})^{2}=(\nabla u^{h})^{2}
\end{equation}
and in particular
\begin{equation}
\frac{1}{h^{2}} \mathcal{E}(u^{h}) \to 0.
\end{equation}
\end{example}

\begin{proof}
The idea of the proof is that we can will construct $u^{h}$ such that $(A^{h})^{2}=(\nabla u^{h})^{2}$. This will make the elastic energy be $0.$
%The idea of the proof is that for small $z,$ we have $\sqrt{\nabla u^{T} \nabla u}(x,y,z) \sim \sqrt{\nabla u^{T} \nabla u}(x,y,0)+z \mathbf{II}(x,y),$ where $\mathbf{II}(x,y)$ is the second fundamental form at the point $(x,y).$ So the problem is now to find a surface such that $\mathbf{II}(x,y)=\begin{bmatrix}
%    1 & 0  \\
%    0 & 0 
%  \end{bmatrix}.$
  
Let 
\begin{equation}
\phi^{h}(x,y)=(\lambda_{h}^{-1}\cos(\lambda_{h} x), \lambda_{h}^{-1}\sin(\lambda_{h} x),y),
\end{equation}
where $\lambda_{h}$ will be determined later. We then have 
\begin{equation}
\widehat{\nu}(x,y)=(\cos(\lambda_{h} x),\sin(\lambda_{h} x),0),
\end{equation}
where $\widehat{\nu}$ is the unit normal to the surface. We also have 
\begin{equation}
\nabla'(\phi^{h})=\begin{bmatrix}
    - \sin(\lambda^{h}x) & \cos(\lambda_{h} x) & 0  \\
    0 & 0 & 1  \\
  \end{bmatrix}^{T}.
\end{equation}

%\begin{equation}
%  \mathbf{II}(x,y)=\begin{bmatrix}
%    \lambda_{h} & 0  \\
%    0 & 0 
%  \end{bmatrix},
%\end{equation}
%where $\mathbf{II}(x,y)$ is the second fundamental form at the point $(x,y).$
  
By taking $\lambda_{h}=\frac{1}{h^{\alpha}}$ and defining $u^{h}:\Omega^{h}\to \mathbf{R}^{3}$ as 
\begin{equation}
u^{h}(x,y,z)=\phi(x,y)+z\widehat{\nu}(x,y), 
\end{equation}
we have
\begin{equation}
\nabla u^{h}=\begin{bmatrix}
-(1+\frac{z}{h^{\alpha}}) \sin(\lambda^{h}x) & 0 &  \cos(\lambda_{h} x)\\
(1+\frac{z}{h^{\alpha}}) \cos(\lambda_{h} x) & 0 & \sin(\lambda^{h}x)\\
0 & 1 & 0
\end{bmatrix}
\end{equation}
and therefore
\begin{equation}
\nabla u=QA^{h},
\end{equation}
where $Q \in O(3).$
%Using the approximations 
%\begin{equation}
%\begin{split}
%(Id+hX)^{-1}&=Id-hX+\mathcal{O}(h^{2})\\
%\sqrt{Id+hX}&=Id+\frac{h}{2}X+\mathcal{O}(h^{2}),
%\end{split}
%\end{equation}
%along with the fact that $\sin(x), \cos(x)$ are uniformly bounded,
We then get
\begin{equation}
\sqrt{A^{h}(x)^{-1}\nabla u^{T}\nabla u{A^{h}(x)^{-1}}}=Id;
\end{equation}
therefore
\begin{equation}
\frac{1}{h}\int_{\Omega^{h}}W(\nabla u^{h}[A^{h}(x)]^{-1})=0.
\end{equation}
%Therefore as long as $4-4\alpha>2$ (i.e. $\alpha<\frac{1}{2}$) we have that
%\begin{equation}
%\frac{1}{h^{2}} \mathcal{E}(\nabla u) \to 0.
%\end{equation}
\end{proof}

%\begin{example}
%Trivial counterexamples: $G=\lambda^{h} Id,$ any multilayer structure with common $G_{2 \times 2}$ but different Cosserat vector.
%\end{example}

The last example may seem pathological since the limiting configuration does not exist (in the strong $W^{1,2}$ sense), i.e. this sequence of minimizers is not compact in the strong $W^{1,2}$ topology. We may ask if adding this additional hypothesis forces the prestrain to be of order $h.$ The next example shows this is not true. The idea is simple: it is to construct $u^{h}$ to form a right-angle corner, and then define $A^{h}$ to make the elastic energy equal to $0.$

\begin{example}
%This example shows that if we do not assume a prestrain of order $h$ then, even if the $W^{1,2}$ limit of minimizers exists, it may not be $W^{2,2}.$

Let $\Omega=[0,1]^{2},$ and let $r=\lambda h$ for a fixed constant $\lambda.$ Let $x_{0}^{-}=\frac{1}{2}(1-\frac{\pi r}{2})$ and $x_{0}^{+}=\frac{1}{2}(1+\frac{\pi r}{2}).$ Let $u^{h}: \Omega^{h} \to \mathbf{R}^{3}$ be defined as
\begin{equation*}
u^{h}(x,y,0)=
\begin{cases} 
      (x,y,0) & x\in [0,x_{0}^{-}] \\
      (x_{0}+r\sin(( \frac{x-x_{0}^{-}}{r} )), y, r(1-\cos(( \frac{x-x_{0}^{-}}{r} ))) & x \in [x_{0}^{-},x_{0}^{+}]\\
       (x_{0}^{-}+r,y,r+x-x_{0}^{+}) & x \in [x_{0}^{+},1],
   \end{cases}
\end{equation*}
and
\begin{equation}
u^{h}(x,y,z)=u^{h}(x,y,0)+z \widehat{\nu}(x,y),
\end{equation}
where $\widehat{\nu}(x,y)$ is the unit normal to the surface parametrized by $u^{h}(x,y,0).$ Let $W(X)= \parallel\sqrt{X^{T}X} - \mbox{Id} \parallel^{2}.$ Then (for $A^{h}$ converging strongly to $Id$ which will be determined in the course of the proof) we have that
\begin{equation}
\frac{1}{h^{2}} \mathcal{E}(u^{h}) \to 0;
\end{equation}
 also that the rescaled deformations converge strongly to a limit $\phi$ which is not is $W^{2,2}.$ In other words, the limit may have corners if the prestrain is not of order $h.$
\end{example}

\begin{proof}
We start by computing $\widehat{\nu}$ for $x \in (x_{0}^{-}, x_{0}^{+}):$ 
\begin{equation}
\nabla \widehat{\nu}(x,y)=
\begin{bmatrix}
r^{-1}\cos\left( \frac{x-x_{0}}{r} \right) & 0 \\
0 & 0 \\
r^{-1} \sin \left( \frac{x-x_{0}}{r}  \right) & 0
\end{bmatrix}
\end{equation}
if $x \in (x_{0}^{-}, x_{0}^{+})$. Therefore for $x \in (x_{0}^{-}, x_{0}^{+})$ we have 
\begin{equation}
\nabla u^{h}(x,y,z)=
\begin{bmatrix}
\cos\left( \frac{x-x_{0}}{r} \right) & 0 & \widehat{\nu}_{1}(x,y)\\
0& 1 & \widehat{\nu}_{2}(x,y)\\
\sin \left( \frac{x-x_{0}}{r} \right) & 0 & \widehat{\nu}_{3}(x,y)
\end{bmatrix}+z\nabla \widehat{\nu}(x,y),
\end{equation}
where $\widehat{\nu}=(\widehat{\nu}_{1}, \widehat{\nu}_{2}, \widehat{\nu}_{3}).$ It's clear that $\phi^{h}: \Omega^{1} \to \mathbf{R}^{3}$ defined as $\phi^{h}(x,y,z)=u^{h}(x,y,hz)$ converges strongly to the function
\begin{equation}
\begin{cases}
\phi(x,y,z)= (x,y,0) & x \in [0,\frac{1}{2}]\\
\phi(x,y,z)=(\frac{1}{2},y,x-\frac{1}{2}) & x \in [0,\frac{1}{2}]
\end{cases},
\end{equation}
and that $\frac{1}{h} \partial_{z} \phi^{h}$ converges strongly in $W^{1,2}$ to 
\begin{equation}
\widehat{\nu}_{0}=\emph{e}_{3} \mathbf{1}_{[0,\frac{1}{2}]}-\emph{e}_{1} \mathbf{1}_{[\frac{1}{2},1]}.
\end{equation}
By basic linear algebra, we can write 
\begin{equation}
\nabla u^{h}(x,y,z)=Q^{h}(x,y,\frac{z}{h})M^{h}(x,y,\frac{z}{h}),
\end{equation}
where $M^{h}\in C^{1}(\Omega^{1}, M^{2\times 2}_{sym})$, and $Q^{h}\in C^{1}(\Omega^{1}, SO(3))$. If we define 
\begin{equation}
A^{h}(x,y,z)=(Q^{h})^{T}(x,y,z)\nabla u^{h}(x,y,hz)
\end{equation}
then 
\begin{equation}
\begin{split}
  (A^{h})^{T}A^{h}&=(\nabla u^{h})^{T}\nabla u^{h}=Id+2z\mbox{sym}\left(\nabla \nu^{T} \begin{bmatrix}
\cos\left( \frac{x-x_{0}}{r} \right) & 0 & \widehat{\nu}_{1}(x,y)\\
0& 1 & \widehat{\nu}_{2}(x,y)\\
\sin \left( \frac{x-x_{0}}{r} \right) & 0 & \widehat{\nu}_{3}(x,y)
\end{bmatrix} \right) + z^{2}(\nabla \nu)^{T}\nabla \nu \\
&\to Id
\end{split}
\end{equation}
in $L^{p}$ for every $p \in [1, \infty)$ since $z \nabla \nu$ is bounded in $L^{\infty}$ in $(x_{0}^{-}, x_{0}^{+})$ and $\nabla \nu =0$ outside of $(x_{0}^{-}, x_{0}^{+}).$ We now claim that 
\begin{equation}
\frac{1}{h^{2}} \mathcal{E}^{h}(u^{h}) \to 0,
\end{equation}
this is easy to see, since
\begin{equation}
W(\nabla u^{h} (A^{h})^{-1})=W(Q^{h})=0,
\end{equation}
%since the functional defined this way is isotropic, i.e.
%\begin{equation}
%W(A)=W(AQ)
%\end{equation}
%for $Q$ orthogonal. 
but clearly $\phi$ is not in $W^{2,2}(\Omega^{1}, \mathbf{R}^{3}).$ This does not contradict theorem \ref{Gammaconvergencetheorem1}, since we have that the prestrain is larger than $h.$

In this example, the limit of $\nabla u^{h}$ exists in the strong sense, and $A^{h} \to Id$ strongly, but the limit is not $W^{2,2}$ this is consistent with the fact that the prestrain is not of order $h.$
\end{proof}

Since in the last example, the limit was not in $W^{2,2,}$ we may ask if adding this hypothesis ensures that the oscillations of the metric are of order $h.$ The following example shows that this not true, as well as exhibiting another abnormality: in Euclidean elasticity, bounded energy up to order $h^{2}$ means that the second fundamental form of the deformation is controlled. Here we give an example where this is not the case, if the prestrain is not of order $h.$ This last example shows that it is possible for all the conclusions of Theorem \ref{Gammaconvergencetheorem1} to hold, with all the hypotheses except a prestrain of order $h.$ 

%Before stating the example, we need a short lemma, which we will use also in the following section.

%\begin{lemma}
%Let $A,M$ be SPD $3 \times 3$ matrices, if $Q \in O(3)$ then
%\begin{equation}
%\parallel A-M \parallel \leq \parallel A-QM \parallel.
%\end{equation}
%\end{lemma}

%\begin{proof}
%Using that SPD matrices are unitarily diagonalizable, and that the Frobenius norm is invariant by multiplying by orthogonal matrices, we may assume $M$ is diagonal. We can write the condition for minimization of $ \parallel A-QM \parallel^{2}$ subject to the constraint $Q \in O(3),$ this reads
%\begin{equation}
%QM \in M_{3\times 3}^{\mbox{sym}}.
%\end{equation}
%
%Hence, $QM$ and $M$ are square roots of $M^{2},$ therefore, if $Q$ is such that $Q$ minimizes $\parallel A-QM \parallel^{2}$ then $Q$ is diagonal with only $\pm 1$ entries:
%\begin{equation}
%Q=\sum a_{i} \mbox{e}_{i}\otimes \mbox{e}_{i},
%\end{equation}
%where $a_{i} \in \{-1,1\}.$ Hence,
%\begin{equation}\label{normofsymmetricmatrix}
%\parallel A-QM \parallel^{2}=\sum (A_{i,j}-\delta_{ij}a_{i}M_{ii})^{2}.
%\end{equation}
%Since $A$ is SPD, we know $A_{ii}, M_{ii} \geq 0.$ Hence, \eqref{normofsymmetricmatrix} is minimized when $a_{i}=1$ and $Q = \mbox{Id}.$ 
%\end{proof}

\begin{example}
Let $u^{h}: \Omega^{h} \to \mathbf{R}^{3}$ be defined as 
\begin{equation}
u^{h}(x,y,0)=(x,y,h^{\alpha} \sin(h^{-\beta} x))
\end{equation}
and 
\begin{equation}
u^{h}(x,y,z)=(x,y,h^{\alpha} \sin(h^{-\beta} x))+z\widehat{\nu}(x,y),
\end{equation}
where $\widehat{\nu}(x,y)$ is the unit normal to the surface parametrized by $u^{h}(x,y,0)$ and $\alpha, \beta$ will be determined later. Let $W(X)=\parallel \sqrt{X^{T}X}-\mbox{Id} \parallel^{2}$ Then, for some $A^{h}$ which will be determined during the proof, we have  
\begin{equation}
\frac{1}{h^{2}} \mathcal{E}^{h}(u^{h}) = 0.
\end{equation}
We also have that, for the rescaled deformation $\overline{u}: \Omega^{1} \to \mathbf{R}^{3},$ that
\begin{equation}
\nabla^{h} \overline{u}^{h} \to \mbox{Id}
\end{equation}
strongly in $L^{2},$ but
\begin{equation}
\frac{1}{h^{3}} \int_{\Omega^{h}} W(\nabla u^{h})dx \to \infty.
\end{equation}
In other words, the sheet has infinite bending. We also have that $A^{h} \to \mbox{Id}$ in $L^{2}$ but
\begin{equation}
\frac{1}{h}\parallel \frac{A^{h}-\mbox{Id}}{h} \parallel_{L^{2}\Omega^{h}}^{2} \to \infty.
\end{equation}
\end{example}

\begin{proof}
We can compute
\begin{equation}
\widehat{\nu}(x,y)=\left(\frac{-h^{\alpha-\beta} \cos(h^{-\beta}x)}{\sqrt{h^{2(\alpha-\beta)} \cos^{2}(h^{-\beta}x)+1}},0, \frac{1}{\sqrt{h^{2(\alpha-\beta)} \cos^{2}(h^{-\beta}x)+1}} \right)
\end{equation}
and then we get 
\begin{equation}
\begin{split}
\nabla u^{h}(x,y,z)=\begin{bmatrix}
    1 & 0 & \frac{h^{\alpha-\beta} \cos(h^{-\beta}x)}{\sqrt{h^{2(\alpha-\beta)} \cos^{2}(h^{-\beta}x)+1}}  \\
    0 & 1 & 0  \\
    h^{\alpha-\beta}\cos(h^{-\beta}x) & 0 &  \frac{1}{\sqrt{h^{2(\alpha-\beta)} \cos^{2}(h^{-\beta}x)+1}}   
  \end{bmatrix} \\
  + z \begin{bmatrix}
    \frac{h^{\alpha-2\beta} \sin(h^{-\beta}x)}{(h^{2(\alpha-\beta)} \cos^{2}(h^{-\beta}x)+1)^{\frac{3}{2}}} & 0 & 0  \\
    0 & 0 & 0  \\
    \frac{h^{2\alpha-3\beta} \sin(h^{-\beta}x)\cos(h^{-\beta}x) }{(h^{2(\alpha-\beta)} \cos^{2}(h^{-\beta}x)+1)^{\frac{3}{2}}} & 0 &  0   
  \end{bmatrix},
  \end{split}
\end{equation}
hence if we define $A^{h}(x)=Q^{h}\nabla(u^{h}(x,y,{z}))$ where $Q^{h}(x,y,{z})$ is such that $Q^{h}\nabla(u^{h}(x,y,{z}))$ is symmetric, then
\begin{equation}\label{zerobendingenergy}
\frac{1}{h^{2}} \mathcal{E}^{h}(u^{h}) = 0.
\end{equation}

We now to check that 
\begin{equation}
A^{h} \to \mbox{Id}.
\end{equation}  
strongly in $L^{2}.$ 
Let 
\begin{equation}
\begin{split}
X&=\begin{bmatrix}
    1 & 0 & \frac{h^{\alpha-\beta} \cos(h^{-\beta}x)}{\sqrt{h^{2(\alpha-\beta)} \cos^{2}(h^{-\beta}x)+1}}  \\
    0 & 1 & 0  \\
    h^{\alpha-\beta}\cos(h^{-\beta}x) & 0 &  \frac{1}{\sqrt{h^{2(\alpha-\beta)} \cos^{2}(h^{-\beta}x)+1}}   
  \end{bmatrix} \\
Y&=\begin{bmatrix}
    \frac{h^{\alpha-2\beta} \sin(h^{-\beta}x)}{(h^{2(\alpha-\beta)} \cos^{2}(h^{-\beta}x)+1)^{\frac{3}{2}}} & 0 & 0  \\
    0 & 0 & 0  \\
    \frac{h^{2\alpha-3\beta} \sin(h^{-\beta}x)\cos(h^{-\beta}x) }{(h^{2(\alpha-\beta)} \cos^{2}(h^{-\beta}x)+1)^{\frac{3}{2}}} & 0 &  0   
  \end{bmatrix}
\end{split}
\end{equation}
provided that $\alpha > 0, \	, \beta> 0,  \alpha -\beta> 0, \	$ we have that
\begin{equation}
\parallel X-\mbox{Id} \parallel=o(1)
\end{equation}
if in addition $\alpha-2\beta \in (-1,0)$ together with $2\alpha-3\beta > -1$ (for example, $\alpha=2.3, \ \beta=1.2$ \footnote{In order to ensure $(A^{h})^{2}$ converges to the identity, we only need $\alpha-2\beta \in (-1,\infty),$ the condition $\alpha-2\beta \in (-1,0),$ will be necessary in what comes later.}, 
then
\begin{equation}
\parallel z Y \parallel=o(1),
\end{equation}
hence
\begin{equation}
\parallel \nabla u^{h}-\mbox{Id} \parallel=o(1).
\end{equation}
Therefore
\begin{equation}
\parallel (\nabla u^{h})^{T}\nabla u^{h}-Id \parallel=\parallel (A^{h})^{2}-Id \parallel= o(1).
\end{equation}
Therefore, by continuity of $\sqrt,$
\begin{equation}
A^{h} \to Id
\end{equation}
in $L^{\infty}$ and in $L^{2}.$
%The conclusion of Theorem \ref{Gammaconvergencetheorem1} holds.

%but the prestrain is not of the same order as the thickness. We clearly have
%\begin{equation}
%A^{h} \to \mbox{Id}
%\end{equation}  
%strongly in $L^{2},$ 

We now claim
\begin{equation}
\frac{1}{h^{3}}\parallel A^{h}-\mbox{Id} \parallel_{L^{2}(\Omega^{h})}^{2} \to \infty.
\end{equation}
In order to check this, we proceed by contradiction: assume that 
\begin{equation}\label{supposeforcontradiction}
\frac{1}{h^{3}}\parallel A^{h}-\mbox{Id} \parallel_{L^{2}(\Omega^{h})}^{2} \leq  C < \infty.
\end{equation}

Note that equation \eqref{zerobendingenergy}, together with equation \eqref{supposeforcontradiction} implies that there exists a measurable rotation field $\widetilde{R}^{h}: \Omega^{h} \to SO(3)$ such that 
\begin{equation}
\parallel \nabla u^{h} - \widetilde{R}^{h} \parallel^{2}_{L^{2}(\Omega^{h})} \leq C h^{3}.
\end{equation}
Proceeding as in the proof of the lower bound (or applying results from \cite{friesecke2002theorem}), we have that there exists a measurable rotation field $R^{h}: \Omega \to SO(3)$ (independent of thickness) such that 
\begin{equation}
\parallel \nabla u^{h} - {R}^{h} \parallel^{2}_{L^{2}(\Omega^{h})} \leq C h^{3}.
\end{equation}
However (using the fact that constants are orthogonal to linear functions in the $(-\frac{1}{2}, \frac{1}{2} )$ interval), we have that for any $R^{h}(x')$
\begin{equation}
\begin{split}
\frac{1}{h^{3}} \parallel \nabla u^{h} - {R}^{h} \parallel^{2}_{L^{2}(\Omega^{h})} &\geq \frac{1}{h^{3}} \int_{\Omega^{h}} \left| z \frac{h^{\alpha-2\beta} \sin(h^{-\beta}x)}{(h^{2(\alpha-\beta)} \cos^{2}(h^{-\beta}x)+1)^{\frac{3}{2}}} \right|^{2} \, dx \, dy \, dz \\
& \to \infty.
\end{split}
\end{equation}

In particular, 
\begin{equation}
\frac{1}{h^{3}}\parallel A^{h}-\mbox{Id} \parallel_{L^{2}(\Omega^{h})}^{2} \to \infty.
\end{equation}
\end{proof}

Since we cannot deduce that the prestrain is of order $h$ from the conclusions, we may ask what additional hypotheses we may add to draw this conclusion. This is done in the following proposition. Before stating it, we introduce a short lemma, which we will need in the proof. 

\begin{lemma}\label{lemmaonsymetricmatrices}
Let $A,M$ be commutative symmetric positive definite $3 \times 3$ matrices, if $Q \in O(3)$ then
\begin{equation}
\parallel A-M \parallel \leq \parallel A-QM \parallel.
\end{equation}
\end{lemma}

\begin{proof}
%Using that SPD matrices are unitarily diagonalizable, and that the Frobenius norm is invariant by multiplying by orthogonal matrices, we may assume $M$ is diagonal.

We write the condition for minimization of $ \parallel A-QM \parallel^{2}$ subject to the constraint $Q \in O(3):$ the EL equation reads
\begin{equation}
\dot{Q}^{T}{Q}+Q^{T}\dot{Q}=0 \Rightarrow \langle A-QM, \dot{Q}M\rangle=0 \Rightarrow \langle Q^{T}( A-QM), Q^{T}\dot{Q}M\rangle=0 .
\end{equation}
The RHS can be simplified to 
\begin{equation}
\mbox{tr}(Q^{T}AMQ^{T}\dot{Q})=\mbox{tr}(M^{2}Q^{T}\dot{Q}),
\end{equation}
which implies 
\begin{equation}
\mbox{tr}(Q^{T}AMQ^{T}\dot{Q})=0,
\end{equation}
since $M^{2}$ is symmetric and $Q^{T}\dot{Q}$ is antisymmetric. This in turn implies
\begin{equation}\label{ELcondition}
Q^{T}AM \in M^{3\times 3}_{sym},
\end{equation}
since the orthogonal complement of symmetric matrices are anisymmetric ones. Hence, $Q^{T}AM$ and $AM$ are symmetric square roots of $A^{2}M^{2},$ since
\begin{equation}
(Q^{T}AM)^{2}=(Q^{T}AM)^{T}(Q^{T}AM)=A^{2}M^{2}.
\end{equation}
Therefore, if $Q$ is such that $Q$ minimizes $\parallel A-QM \parallel^{2}$ then $Q^{T}AM$ and $AM$ have the same polar decomposition, except for possibly the sign of the eigenvalues. In other words, $Q$ is of the form
\begin{equation}
Q=R^{T}LR,
\end{equation}
where $R$ is such that
\begin{equation}
A=R^{T}DR,
\end{equation}
with $D$ diagonal, and $L$ diagonal with only $\pm 1$ entries:
\begin{equation}
L=\sum a_{i} \mbox{e}_{i}\otimes \mbox{e}_{i},
\end{equation}
where $a_{i} \in \{-1,1\}.$ Hence,
\begin{equation}\label{normofsymmetricmatrix}
\parallel A-QM \parallel^{2}=\sum (\lambda_{i}-a_{i}\gamma_{i})^{2},
\end{equation}
where $\lambda_{i}, \gamma_{i}$ are the eigenvalues of $A$ and $M$ respectively. Since $A$ and $M$ are SPD, we know $\lambda_{i}, \gamma \geq 0.$ Hence, \eqref{normofsymmetricmatrix} is minimized when $a_{i}=1$ and $Q = \mbox{Id}.$ 
\end{proof}

\begin{remark}
The hypothesis that $A$ and $M$ commute is necessary, since otherwise, $AM$ is not symmetric in general, and we can decompose it as
\begin{equation}
AM=QR,
\end{equation}
where $Q$ is in $SO(3)$ and $R$ is symmetric. Hence, $Q$ satisfies the EL equation since $Q^{T}AM$ is symmetric. Therefore, if $A$ and $M$ do not commute, $Q=Id$ does not satisfy eq. \ref{ELcondition} since $AM$ is not symmetric. 
\end{remark}

We now state the desired proposition:

\begin{proposition} \label{prestrainoforderh}
Assume that
\begin{equation}\label{zerobendingenergy2}
\frac{1}{h^{2}} \mathcal{E}^{h}(u^{h}) \leq C < \infty,
\end{equation}
with $A^{h}(x',x_{3}) \to A(x')$ in $L^{2},$ and that $A(x')$ satisfies $\textbf{x}A\textbf{x}^{T} \geq c \parallel \textbf{x} \parallel^{2}.$ Assume further that there exists a rotation field $R(x')$ such that
\begin{equation}\label{uhclosetoA}
\parallel \nabla u^{h}A^{-1}(x')-R(x') \parallel_{L^{2}(\Omega^{h})}^{2} \leq Ch^{3},
\end{equation}
then 
\begin{equation}
\parallel (A^{h})^{T}A^{h}- (A(x'))^{T}A(x') \parallel_{L^{1}(\Omega^{h})} \leq Ch^{2}.
\end{equation}
Furthermore, if $A(x')$ and $A^{h}(x)$ commute, then
\begin{equation}
\parallel A(x')-A^{h}(x) \parallel_{L^{2}(\Omega^{h})}^{2} \leq Ch^{3}.
\end{equation}
\end{proposition}

\begin{proof}
We start by noting that \eqref{zerobendingenergy2} implies that for some measurable $R: \Omega^{1} \to SO(3)$ we have
\begin{equation}\label{uhclosetoAh}
\parallel \nabla u^{h}-RA^{h} \parallel_{L^{2}(\Omega^{h})}^{2} \leq C h^{3},
\end{equation}
which implies
\begin{equation}
\parallel (A^{h})^{T}A^{h}-(\nabla u^{h})^{T}\nabla u^{h} \parallel_{L^{1}(\Omega^{h})} \leq Ch^{2} 
\end{equation}\footnote{Here we have used a result that, though elementary, may deserve an explanation: If $f_{n}, g_{n} \in L^{2}(\Omega^{h}, M_{n\times n})$ are such that $\parallel f_{n}-g_{n} \parallel_{L^{2}(\Omega^{h})}^{2} \leq C h^{3}$ and $\parallel g_{n} \parallel_{L^{2}(\Omega^{h})}^{2}+\parallel f_{n} \parallel_{L^{2}(\Omega^{h})}^{2}\leq Ch$ then
\begin{equation}
\begin{split}
 \parallel f_{n}^{T}f_{n}-g_{n}^{T}g_{n} \parallel_{L^{1}(\Omega^{h})}&=\parallel f_{n}^{T}f_{n}-f_{n}^{T}g_{n}+f_{n}^{T}g_{n}-g_{n}^{T}g_{n} \parallel_{L^{1}(\Omega^{h})}\\
 &\leq \parallel f_{n}^{T}f_{n}-f_{n}^{T}g_{n}\parallel_{L^{1}(\Omega^{h})}+ \parallel f_{n}^{T}g_{n}-g_{n}^{T}g_{n} \parallel_{L^{1}(\Omega^{h})}\\
 &\leq  \parallel f_{n}^{T}(f_{n}-g_{n})\parallel_{L^{1}(\Omega^{h})}+ \parallel (f_{n}^{T}-g_{n}^{T})g_{n} \parallel_{L^{1}(\Omega^{h})}\\
 &\leq Ch^{2}
 \end{split}
 \end{equation}
}
and also that
\begin{equation}
\parallel (A(x'))^{T}A(x')-(\nabla u^{h})^{T}\nabla u^{h} \parallel_{L^{1}(\Omega^{h})} \leq Ch^{2},
\end{equation}
hence 
\begin{equation} \label{orderhdeviationofprestrain}
\parallel (A^{h})^{T}A^{h}- (A(x'))^{T}A(x')\parallel_{L^{1}(\Omega^{h})} \leq Ch^{2}.
\end{equation} 
To check the second part if $A^{h}$ and $A$ commute, we simply have to note that \eqref{uhclosetoAh} and \eqref{uhclosetoA} imply
\begin{equation}
\parallel A-RA^{h} \parallel_{L^{2}(\Omega^{h})}^{2} \leq h^{3},
\end{equation}
which, together with Lemma \ref{lemmaonsymetricmatrices} implies
\begin{equation}
\parallel  A-A^{h} \parallel_{L^{2}(\Omega^{h})}^{2} \leq h^{3}.
\end{equation}

% we proceed by contradiction: assume that
%\begin{equation}
%\frac{1}{h} \parallel A^{h}-A \parallel_{L^{2}} \to \infty.
%\end{equation}
%We present the proof for the case $A^{h} \to A$ in $L^{\infty},$ with the general case requiring only minor adjustments. We start by rewriting
%\begin{equation}
%\begin{split}
%\parallel (A^{h})^{2}-A^{2} \parallel_{L^{1}}&=\parallel (A^{h}-A)(A^{h}-A+2A) \parallel_{L^{1}}\\
%&\geq C \parallel A^{h}-A \parallel_{L^{1}}\\
%&\geq c \parallel A^{h}-A \parallel_{L^{2}}.
%\end{split}
%\end{equation}
%In the last lines, we have used Holder's inequality, and the fact that $2A+A-A^{h}$ is uniformly positive definite since $A-A^{h} \to 0$ in $L^{\infty}$.
\end{proof}

\begin{remark}
We may ask if Proposition \ref{prestrainoforderh} is still true under a weaker hypothesis, like 
\begin{equation}
\parallel \overline{\nabla u}^{h}A^{-1}(x')-R(x') \parallel_{L^{2}(\Omega^{h})} \leq Ch^{3},
\end{equation}
where
\begin{equation}
\overline{\nabla u}^{h}=\frac{1}{h}\int_{-\frac{h}{2}}^{\frac{h}{2}} \nabla u^{h} dz, 
\end{equation}
but the simple example $\Omega=(0,1)\times (0,1)$ and 
\begin{equation}
u^{h}(x,y,z)=(x,y,z)+zh^{\alpha}\sin(h^{-\beta}x)\sin(h^{-\beta}y)
\end{equation}
for the right choice of $\alpha$ and $\beta$ (for example $\alpha=2$ and $\beta=3$) shows that this is not true.
\end{remark}

%\begin{remark}
%This example shows that, even though a prestrain of order $h$ is not a necessary condition in theorem [GIVEN REFERENCE], it is useful to rule out pathological behavior. 
%\end{remark}

Lastly, we prove that if a sheet is in the bending regime, i.e. if
\begin{equation}
\frac{1}{h^{2}}\mathcal{E}^{h}\left( u^{h} \right) \leq C < \infty
\end{equation} 
and the limiting configuration exists in the strong $W^{1,2}$ sense, the limit of the prestrain is thickness independent. 

\begin{proposition}\label{propositiongraduhconvergestosomething}
Let $A^{h}(x) \to A(x)$ strongly in $L^{2}(\Omega^{1})$ (we abused notation by stating this for $A^{h}$ and not the re-scaled versions). Assume that there exists $X \in L^{2}(\Omega^{1} \to M^{3 \times 3})$ such that $\nabla^{h}\widetilde{u}^{h} \to X$ strongly, where $\widetilde{u}^{h}: \Omega^{1} \to \mathbf{R}^{3}$ is the rescaled version of $u^{h}$, and assume that
\begin{equation}
\frac{1}{h^{2}} \mathcal{E}^{h}(u^{h}) \leq C < \infty. 
\end{equation} 
Then $(A(x))_{2\times 2}$ is independent of $x_{3}$
% and 
%\begin{equation}
%\parallel \frac{(B^{h}(x))_{2\times 2}}{h}\parallel_{L^{2}} \leq k < \infty.
%\end{equation}
Furthermore, if $X$ is independent of $x_{3}$ then $A(x)$ is independent of $x_{3}$
% and 
%\begin{equation}
%\parallel \frac{B^{h}(x)}{h}\parallel_{L^{2}} \leq k < \infty
%\end{equation}
\end{proposition}

\begin{proof}
Note that under these hypotheses, we can write
\begin{equation}
A^{h}(x)=A(x)+B^{h}(x),
\end{equation}
where $B^{h}(x) \to 0$ in $L^{2}.$ Since $\mathcal{E}^{h}\leq Ch^{2},$ we know that there exist measurable rotation fields $R^{h}: \Omega^{1} \to SO(3)$ such that
\begin{equation}
\parallel R^{h}\nabla^{h}\widetilde{u}^{h}-A^{h} \parallel_{L^{2}(\Omega^{1})}\leq Kh^{2}.
\end{equation} 
From this we know that $R^{h}\nabla^{h}\widetilde{u}^{h} \to A(x)$ in $L^{2},$ and hence
\begin{equation}
\begin{split}
A^{T}A&=\lim_{h\to 0} (R^{h}\nabla^{h}\widetilde{u}^{h})^{T}(R^{h}\nabla^{h}\widetilde{u}^{h})\\
&=\lim_{h\to 0} (\nabla^{h}\widetilde{u}^{h})^{T}(\nabla^{h}\widetilde{u}^{h})\\
&= X^{T}X,
\end{split}
\end{equation}
where the limit is in the $L^{1}$ topology \footnote{We have used that if $f^{h} \to f$ in $L^{2}$ then $f^{h}(f^{h})^{T} \to ff^{T}$ in $L^{1}.$}. The first and second columns of $X$ are independent of $x_{3}$ (we will prove this in a moment). We therefore have that $(A^{T}A)_{2\times 2}$ is independent of $x_{3},$ and if $X$ is independent of $x_{3},$ then $A^{T}A$ is independent of $x_{3}.$

To prove that the first and second columns of $X$ are independent of $x_{3},$ note that since 
\begin{equation}
\parallel \nabla^{h}\widetilde{u}^{h} \parallel_{L^{2}} \leq C,
\end{equation}
we have that 
\begin{equation}
\parallel \partial_{3}\widetilde{u}^{h} \parallel_{L^{2}} \to 0,
\end{equation}
hence $\widetilde{u}^{h} \to \widetilde{u}$ in $W^{1,2}(\Omega^{1})$ for some $\widetilde{u}$ that satisfies 
\begin{equation}
\partial_{3} \widetilde{u} =0,
\end{equation}
from which we immediately get that $(\nabla \widetilde{u})_{3\times 2}$ is independent of $x_{3},$ since $X_{3\times 2}=(\nabla \widetilde{u})_{3\times 2},$ we have that the first and second columns of $X$ are independent of $x_{3}.$

%To prove that the first and second columns of $X$ are independent of $x_{3},$ note that if $\nabla^{h}u^{h}$ is bounded, we can write
%\begin{equation}
%u^{h}=\widetilde{u}(x')+h(\alpha_{h}(x',x_{3})),
%\end{equation}
%where 
%\begin{equation}
%\widetilde{u}(x')=\int_{-\frac{1}{2}}^{\frac{1}{2}}u(x',t)dt,
%\end{equation}
% and 
% \begin{equation}
% \parallel \nabla^{h}\alpha^{h} \parallel_{L^{2}} \leq k < \infty,
% \end{equation}
%hence,
%\begin{equation}
%(\partial_{1} u^{h}|\partial_{2} u^{h}) \to \nabla'\widetilde{u}(x')
%\end{equation}

\end{proof}

\begin{remark}
%The hypothesis $A^{h} \to A$ instead of $A^{h} \rightharpoonup A$ is necessary to ensure that the limit realizes the metric, i.e. that $\lim (A^{h})^{T}A^{h}=A^{T}A.$ This means ruling out a "wrong" choice of rotation. For example, if $G(x)=Id,$ then any rotation valued field $R(x)$ satisfies that $R^{T}R=Id$. Hence it's possible that $R(x) \rightharpoonup 0.$

It is not possible to weaken the hypotheses to $A^{h} \rightharpoonup A,$ since in general this does not imply $\lim_{h \to 0} (A^{h})^{T}A^{h}=A^{T}A.$
\end{remark}

\begin{remark}
The simple example
\begin{equation}
A^{h}(x',x_{3})=
\begin{cases}
\begin{bmatrix}
    1 & 0 & 0  \\
    0 & 1 & 0  \\
    0 & 0 & 1 
  \end{bmatrix} \mbox{ if } x_{3}<0\\
\begin{bmatrix}
    1 & 0 & 0  \\
    0 & 1 & 0  \\
    0 & 0 & 2 
  \end{bmatrix} \mbox{ if } x_{3} \geq 0
\end{cases}
\end{equation}
show that, in general, the third block and column of $A$ may depend on $x_{3},$ even if the other hypotheses of Proposition \ref{propositiongraduhconvergestosomething} are satisfied.
\end{remark}

\section*{Acknowledgments}

The author wants wants to thank his PhD adviser, Prof. Robert Kohn for many useful discussions. This work does not have any conflicts of interest. This work was partially supported by the National Science Foundation through grant DMS-1311833.

\bibliographystyle{apacite}
\bibliography{References}
\end{document}